\documentclass[11pt]{article}
\usepackage[utf8]{inputenc}
\usepackage{fullpage}
\usepackage{amsthm}
\usepackage{complexity}
\usepackage{mathtools,amsmath,amssymb}
\usepackage{graphicx}
\usepackage{color}
\usepackage{multirow}
\usepackage{tabularx}
\usepackage[colorinlistoftodos]{todonotes}
\usepackage[colorlinks=true, allcolors=blue, backref=page]{hyperref}
\usepackage{enumitem}
\usepackage{float}

\usepackage{thmtools}
\usepackage{thm-restate}

\RequirePackage[colorlinks=true]{hyperref}
\hypersetup{
  linkcolor=[rgb]{0,0,0.5},
  citecolor=[rgb]{0, 0.5, 0},
  urlcolor=[rgb]{0.5, 0, 0}
}

\usepackage[breakable]{tcolorbox}
\newtcolorbox{sgd}[2][]
{
  breakable,
  colframe = gray!50,
  colback  = gray!10,
  coltitle = gray!10!black,
  before skip = 10pt,
  after skip = 10pt,
  title    = \textbf{#2.},
  #1,
}
\newtcolorbox{pca}[2][]
{
  breakable,
  colframe = red!50,
  colback  = red!10,
  coltitle = red!10!black,
  before skip = 10pt,
  after skip = 10pt,
  title    = \textbf{Example: Local convergence for streaming $k$-PCA #2},
  #1,
}

\usepackage{collectbox}



\usepackage{algorithm2e}

\newif\ifdraft
\drafttrue

\usepackage{amsthm}
\usepackage{thmtools,thm-restate}

\numberwithin{equation}{section}
\declaretheoremstyle[bodyfont=\it,qed=\qedsymbol]{noproofstyle}

\declaretheorem[name=Observation,numbered=no]{observation*}

\declaretheorem[numberlike=equation]{theorem}

\declaretheorem[name=Theorem,numbered=no]{theorem*}

\declaretheorem[numberlike=equation]{lemma}
\declaretheorem[name=Lemma,numbered=no]{lemma*}

\declaretheorem[name=Corollary,numbered=no]{corollary*}

\declaretheorem[name=Proposition,numbered=no]{proposition*}

\declaretheorem[name=Claim,numbered=no]{claim*}

\declaretheorem[name=Conjecture,numbered=no]{conjecture*}

\declaretheorem[name=Question,numbered=no]{question*}

\declaretheoremstyle[bodyfont=\it]{defstyle} 

\declaretheorem[numberlike=equation,style=defstyle]{definition}
\declaretheorem[unnumbered,name=Definition,style=defstyle]{definition*}

\declaretheorem[unnumbered,name=Example,style=defstyle]{example*}

\declaretheorem[unnumbered,name=Notation=defstyle]{notation*}

\declaretheorem[unnumbered,name=Construction,style=defstyle]{construction*}

\declaretheoremstyle[]{rmkstyle}

\newcommand{\N}{\mathbb{N}}
\newcommand{\Nz}{\mathbb{N}_{\geq0}}
\newcommand{\Real}{\mathbb{R}}
\newcommand{\Exp}{\mathbb{E}}
\newcommand{\Var}{\mathsf{Var}}

\newcommand{\gap}{\textsf{gap}}
\newcommand{\diag}{\textsf{diag}}
\newcommand{\tr}{\textsf{tr}}

\newcommand{\bg}{\mathbf{g}}
\newcommand{\bw}{\mathbf{w}}

\newcommand{\bx}{\mathbf{x}}

\newcommand{\cD}{\mathcal{D}}

\newcommand{\cF}{\mathcal{F}}
\newcommand{\cW}{\mathcal{W}}

\DeclarePairedDelimiter\ceil{\lceil}{\rceil}

\makeatletter
\def\moverlay{\mathpalette\mov@rlay}
\def\mov@rlay#1#2{\leavevmode\vtop{%
   \baselineskip\z@skip \lineskiplimit-\maxdimen
   \ialign{\hfil$\m@th#1##$\hfil\cr#2\crcr}}}
\newcommand{\charfusion}[3][\mathord]{
    #1{\ifx#1\mathop\vphantom{#2}\fi
        \mathpalette\mov@rlay{#2\cr#3}
      }
    \ifx#1\mathop\expandafter\displaylimits\fi}
\makeatother

\author{
Chi-Ning Chou\thanks{School of Engineering and Applied Sciences, Harvard University, Cambridge, Massachusetts, USA. Supported by NSF awards CCF 1565264 and CNS 1618026. Email: \texttt{chiningchou@g.harvard.edu}.}
\and Juspreet Singh Sandhu\thanks{School of Engineering and Applied Sciences, Harvard University, Cambridge, Massachusetts, USA. Supported by DARPA ONISQ program award HR001120C0068. Email: \texttt{jus065@g.harvard.edu}.}
\and Mien Brabeeba Wang\thanks{MIT CSAIL, Cambridge, Massachusetts, USA. Supported by NSF Awards CCF-1810758, CCF-0939370, CCF-1461559 and Akamai Presidential Fellowship. Email: \texttt{brabeeba@mit.edu}.}
\and Tiancheng Yu\thanks{MIT LIDS, Cambridge, Massachusetts, USA. Supported by NSF BIGDATA grant 1741341. Email: \texttt{yutc@mit.edu}.}}

\title{A General Framework for Analyzing Stochastic Dynamics in Learning Algorithms}

\begin{document}
\maketitle
\begin{abstract}
One of the challenges in analyzing learning algorithms is the circular entanglement between the objective value and the stochastic noise. This is also known as the ``chicken and egg'' phenomenon and traditionally, there is no principled way to tackle this issue. People solve the problem by utilizing the special structure of the dynamic, and hence the analysis would be difficult to generalize.

In this work, we present a streamlined three-step recipe to tackle the ``chicken and egg'' problem and give a general framework for analyzing stochastic dynamics in learning algorithms. Our framework composes standard techniques from probability theory, such as stopping time and martingale concentration. We demonstrate the power and flexibility of our framework by giving a unifying analysis for three very different learning problems with the last iterate and the strong uniform high probability convergence guarantee. The problems are stochastic gradient descent for strongly convex functions, streaming principal component analysis, and linear bandit with stochastic gradient descent updates. We either improve or match the state-of-the-art bounds on all three dynamics.
\end{abstract}
\thispagestyle{empty}

\newpage
\setcounter{tocdepth}{2}
\tableofcontents
\thispagestyle{empty}

\newpage
\clearpage
\pagenumbering{arabic} 
\section{Introduction}\label{sec:intro}
Iterative methods are widely used in machine learning and stochastic optimization where the objective functions naturally induce \textit{stochastic processes}. For example, when an algorithm uses stochastic gradient descent (SGD) updates, the value of the loss function forms a stochastic process. Therefore, to study the performance of a learning algorithm, it usually suffices to understand the \textit{behavior} of the corresponding stochastic process.

There have been many successes in providing theoretical guarantees for various learning algorithms. However, as the learning algorithms nowadays become increasingly complicated, it is more and more challenging to conduct clean and tight theoretical analysis. Moreover, due to the lack of general principles for analysis, the existing theoretical studies are usually tailored to specific learning dynamics and hence are difficult to extend to other problems. One main challenge of the analysis, sometimes known as a \textit{``chicken and egg problem''}~\cite{harvey2019tight}, originates from the circular relationship between the improvement of the process and the historical stochasticity. This problem is ubiquitous and often causes the analysis to be complicated, ad hoc, and sub-optimal (see~\autoref{sec:running chicken and egg} and~\autoref{sec:SGD comparison}).

In this work, we propose a framework with an attempt to address the lack of unifying theoretical analysis for stochastic processes and the chicken and egg problem in learning algorithms. Before introducing our framework and comparing it with the previous works, let us start with a warm-up puzzle to emphasize the subtlety of analyzing stochastic processes.

\subsection{A warm-up puzzle about conditional expectation}
Consider a dice with $6$ faces, what is the expected number of times you need to roll this die before you get the number $1$? It's not hard to see that the expected number would be $6$. Now, let us slightly change the question by asking what is the expected number of times you need to roll this die before you get the number $1$
\textit{conditioned} on the event that you have only seen odd number in the process? Most people's first reaction would be $3$ because it seems that there are only $3$ possible numbers left after the conditioning, however, the answer is actually $3/2$. 

The reason why $3$ is not the correct answer is because the most natural way to calculate the conditional expectation is wrong. In particular, directly writing down an expression conditioning on \textit{a probability event} would mislead one to condition on a wrong probability space. 
It turns out that a more principle way to reasoning about conditioning (in stochastic processes) is to properly utilize the corresponding \textit{stopping time}. In fact, this is one of the main messages we want to convey in this paper: ``stopping time and stopped process elegantly reveal the structure of a probability space after conditioning''. This will be clearer after we define these concepts formally and apply them in the framework. Also, see~\autoref{app:puzzle} for more details on this puzzle.

Next, let us set up a running example for illustrating the chicken and egg problem and our framework.

\subsection{A running example: SGD for strongly convex functions}\label{sec:running example}
We pick stochastic gradient descent (SGD) for strongly convex functions as a running example because it is one of the simplest (but non-trivial) and most common methods in machine learning and optimization. While being well-studied, tight high-probability convergence guarantees of SGD for strongly convex functions turn out to be more complicated than expected. In particular, the previous analysis~\cite{rakhlin2012making,harvey2019tight,jain2019making} for the two common convergence types are quite different and specialized tools are required. The main challenges lie in the \textit{chicken and egg problem}, which we will explain after setting up the problem in the following.

\begin{sgd}{Setup of SGD for strongly convex functions}
Let $F$ be a convex function over a convex domain $\cW$ equipped with norm $\|\cdot\|$. We say $F$ is $\lambda$-strongly convex for $\lambda>0$ if for all $\bw,\bw'\in\cW$ and a subgradient $\bg$ at $\bw$,
\[
F(\bw')\geq F(\bw)+\bg^\top(\bw'-\bw)+\frac{\lambda}{2}\|\bw'-\bw\|^2 \, .
\]
Since $F$ is $\lambda$-strongly convex, there exists a unique $\bw^*$ that attains the minimum value of $F$.
For $G > 0$, a $G$-bounded stochastic gradient oracle on input $\bw$ returns a random $\hat{\bg}$, such that $\Exp[\hat{\bg}]$ is a subgradient of $F$ at $\bw$ and $\|\hat{\bg}\|\leq G$ almost surely. 

\vspace{0.5cm}
\begin{algorithm}[H] 
	\caption{SGD for strongly convex function}\label{alg:SGD for strongly convex}
	    \textbf{Input:} Time parameter $T\in\N$ and step size parameters $\{\eta_t\}_{t\in\N}$.\\
	    \textbf{Initialize:} $\bw_0\in\cW$.\\
	    \For{$t=1,\dots,T$}{
	    Query the stochastic gradient oracle of $F$ at $\bw_{t-1}$ and get $\hat{\bg}_t$.\\
	    Let $\bw_t=\textsf{Proj}_{\cW}(\bw_{t-1}-\eta_t\hat{\bg}_t)$ where $\textsf{Proj}_{\cW}$ is an orthogonal projection operator for $\cW$.}
	    \textbf{Output:} $\bw_T$.
\end{algorithm}
\vspace{0.5cm}
The SGD algorithm (see~\autoref{alg:SGD for strongly convex}) minimizes $F$ over $\cW$ by maintaining a vector $\bw_t$ at time $t$ and updating it with a stochastic gradient oracle weighted by a tunnable learning rate parameter $\eta_t$. The goal of the algorithm is to have $\bw_t\rightarrow\bw^*$ as $t\rightarrow\infty$.

The objective function for analyzing the SGD algorithm is naturally defined as $X_t := \|\bw_t-\bw^*\|^2$ and we would like to show that $X_t$ converges to $0$ with high probability. 
\end{sgd}

In the previous work, there are two types of high probability convergence guarantee that are of interest: the last iterate convergence and the strong uniform convergence. While Harvey et al.~\cite{harvey2019tight} and Rakhlin et al.~\cite{rakhlin2012making} proved either optimal or nearly optimal convergence rate for the two guarantees respectively, their analysis relied on specialized techniques and turned out to be quite different (see~\autoref{sec:SGD comparison} for a detailed comparison). In this work, we apply our framework and give a simple and unifying proof for the two convergence guarantees.

\begin{theorem}[Convergence of SGD for strongly convex functions]\label{thm:sgd intro}
Consider the above setting with learning rate $\eta_t=1/(\lambda t)$ for every $t\in\N$. For every $\delta>0$, we have
\begin{itemize}
\setlength\itemsep{0mm}
\item (last iterate convergence) $\Pr[X_T>\frac{500G^2\log(1/\delta)}{\lambda^2\cdot T}]<\delta$ for every $T\in\N$ and
\item (strong uniform convergence) $\Pr[\exists t\in\N,\ X_t>\frac{1000G^2(\log(1/\delta)+\log\log(t+1))}{\lambda^2\cdot t}]<\delta$.
\end{itemize}
\end{theorem}

For the simplicity of exposition, we do not optimize the constants in the analysis and focus on the last iterate convergence with parameters $\lambda=G=1$ in the rest of the introduction. Also, we focus on the scaling in the convergence rate and hence the constants are not optimized and are rounded to nice-looking integers.

\subsection{The chicken and egg problem}\label{sec:running chicken and egg}
In this paper, the chicken and egg problem refers to the situation where the improvement of the process is entangled with historical stochasticity. This often causes complications in the analysis and people usually developed specialized tools to tackle the specific problems they are interested in. To be more concrete, let us take a look at the chicken and egg problem in our running example. We start with an expectation analysis that shows $\Exp[X_T]=O(1/T)$.

\begin{sgd}{Expectation analysis for SGD}
Recall that the goal is to show that $X_T$ converges to $0$ with high probability. Due to the iterative nature of the process, the usual first step of the analysis is to find a recursion for $X_T$, i.e., upper bound $X_T$ with a function of $X_{t}$ for some $t<T$. In the running example, it is standard to rewrite the dynamic as follows. For every $t\geq3$,
\begin{align}
X_t&=\|\bw_t-\bw^*\|^2=\|\textsf{Proj}_{\cW}(\bw_{t-1}-\hat{\bg}_t/t)-\bw^*\|^2 \nonumber \leq \|\bw_{t-1}-\hat{\bg}_t/t-\bw^*\|^2 \nonumber\\
&=\left(1-\frac{2}{t}\right)\cdot X_{t-1}+\frac{2\left(X_{t-1}-\hat{\bg}_t^\top(\bw_{t-1}-\bw^*)\right)}{t}+\frac{\|\hat{\bg}_t\|^2}{t^2} \, . \label{eq:sgd recursion}
\end{align}
By taking expectation on both sides, for every $T\geq 3$, we have
\begin{align}
\Exp[X_T]&\leq\Exp\left[\left(1-\frac{2}{T}\right)\cdot X_{T-1}+\frac{2\left(X_{T-1}-\hat{\bg}_T^\top(\bw_{T-1}-\bw^*)\right)}{T}+\frac{\|\hat{\bg}_T\|^2}{T^2}\right] \label{eq:running example 1}\\
&\leq \left(1-\frac{2}{T}\right)\cdot\Exp[X_{T-1}]+\frac{1}{T^2} \nonumber\leq \frac{2}{T(T-1)}\cdot\Exp[X_2] + \sum_{t'=3}^T \frac{t'(t'-1)}{T(T-1)}\cdot\frac{1}{t'^2}=O\left(\frac{1}{T}\right) \nonumber
\end{align}
where the second inequality uses the strong convexity and the last approximation uses the boundedness property of the SGD oracle\footnote{A $G$-bounded SGD oracle for $\lambda$-strongly convex function guarantees $X_t\leq4G^2/\lambda^2$ for all $t\in\N$.}.
\end{sgd}

The chicken and egg problem shows up when we move on to proving high probability convergence guarantees.

\begin{sgd}{The chicken and egg problem in SGD}
Note that the second term in the right hand side of~\autoref{eq:running example 1} no longer disappears as it did in the expectation analysis. Nevertheless, we can still mimic the analysis by letting $M_{t}=\sum_{t'=3}^tt'(t'-1)\cdot(2\left(X_{t'-1}-\hat{\bg}_{t'}^\top(\bw_{t'-1}-\bw^*)\right)/t'+\|\hat{\bg}_{t'}\|^2/t'^2)$ for every $t\geq 3$ and obtain
\begin{equation}\label{eq:running example 2}
X_t \leq \frac{2}{t(t-1)}\cdot\left(X_{2} + M_{t}\right) \, ,\ \forall t\geq3 \, .
\end{equation}
To show $X_T=O(1/T)$ with high probability, by~\autoref{eq:running example 2} it suffices to show that $M_T=O(T)$ with high probability. However, to prove such bound for $M_T$, we also need to have a tight control on $X_{t}$ for every $t<T$ because the stochasticity in $M_T$ depends on $X_{t}$. In particular, a direct application of concentration inequality (e.g., Lemma~\ref{lem:martingale concentration}) requires $X_{t}=O(1/t)$.
Meanwhile, to properly analyze $X_{t}$, we again need to understand $M_{t}$ and this brings us back to where we were in the beginning. We tokenize such causality dilemma between showing $X_t=O(1/t)$ and bounding $M_t=O(t)$ as the chicken and egg problem in analyzing SGD for strongly convex functions. See~\autoref{fig:cgp in SGD} for a pictorial explanation. 

\begin{figure}[H]
    \centering
    \includegraphics[width=6cm]{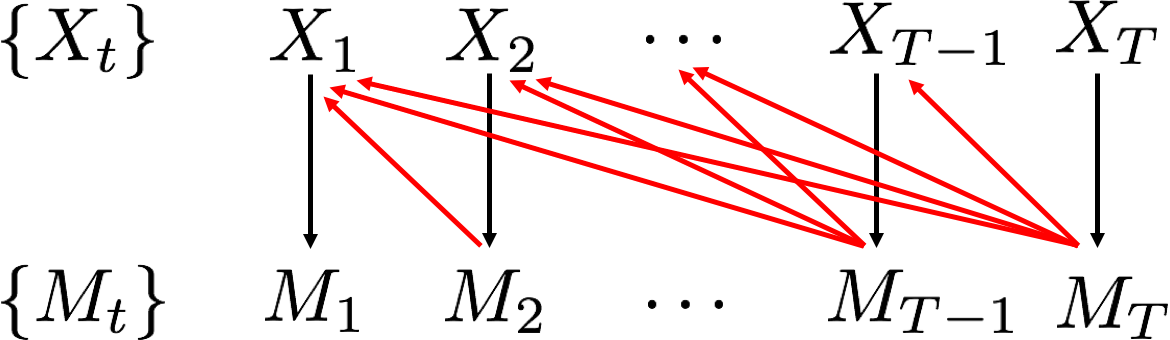}
    \caption{The chicken and egg problem in SGD. An arrow from $A$ to $B$ means that proving an upper bound for $A$ requires an upper bound for $B$.}
    \label{fig:cgp in SGD}
\end{figure}
\end{sgd}

\paragraph{Previous approach to resolve the chicken and egg problem.}
To overcome the ``chicken and egg'' problem, previous analysis usually either (i) only looks at the expectation and loses important local information (e.g.,~\cite{chen2019convergence,qiao2019rates,yu2019double,zou2019sufficient}), (ii) directly analyzes the process step-by-step which can cause overcomplication (e.g.,~\cite{AL17,allen2019convergence,richards2019optimal}), or (iii) focuses on a specific problem setting and hence is not easy to generalize.

\subsection{Tackling the chicken and egg problem using stopping time}

We develop a handy tool that systematically tackles the chicken and egg problem. The key idea is to design a \textit{stopping time} $\tau$ with the advantage that the stopped process\footnote{The notation $\cdot\wedge\cdot$ stands for taking the minimum of the two sides, i.e., $X_{t\wedge\tau}=\mathbf{1}_{t<\tau}X_t+\mathbf{1}_{t\geq\tau}X_{\tau}$.} $\{X_{t\wedge\tau}\}$. Mathematically, the stopped process $\{X_{t\wedge\tau}\}$ simulates the original process $\{X_t\}$ but stops at a certain value whenever some probability events happen. See~\autoref{sec:prelim} for a formal definition. For now, intuitively we use a stopped process to capture the typical behaviors of the original process $\{X_t\}$ and rule out the undesirable atypical events.
As a consequence, we can prove tighter bounds on the moment information of the stopped process and thus get a better convergence rate.

Now that we have shown the convergence of the stopped process, what can we say about the original process? We use a \textit{pull-out lemma} which provides a sufficient condition to extend the convergence guarantee from $\{X_{t\wedge\tau}\}$ to $\{X_t\}$ without introducing any extra factors to the convergence rate. 
Informally, the pull-out condition is of the following form:
\begin{lemma}[Informal,~\cite{CW20}]
Let $\{X_t\}_{t\in\N}$ be a stochastic process and let $\tau$ be a stopping time. For every $\delta\in(0,1)$, let $\{E_t\}_{t\in\N}$ be a sequence of probability (error) events such that $\Pr[\exists t\in\N,\ X_{t\wedge\tau}\in E_t]<\delta$. Suppose the following condition
\begin{align*}
&\text{``if the event $\{E_t\}$ does not happen up to time $t$}\Rightarrow\ \text{the stopping time $\tau$ won't stop at time $t$''}
\end{align*}
holds for every $t\in\N$, then we also have $\Pr[\exists t\in\N,\ X_t\in E_t]<\delta$.
\end{lemma}

While this may sound magical, it is actually not that surprising because this is exactly what the stopping time $\tau$ is designed for --- the stopped process typically would behave the same as the original process and hence the error probability should be upper bounded by that of the stopped process. See~\autoref{sec:framework improvement} and~\autoref{sec:SGD pull out} for more details.

\subsection{A general framework for analyzing stochastic dynamics}\label{sec:examples}

It turns out that our stopping time tool can handle problems much beyond the running example. We build up a framework to tackle the chicken and egg problems in stochastic processes and provide a recipe on how to apply it in~\autoref{sec:framework}. 
We would like to stress in advance that the technical tools in our framework are basic and it is the \textit{composition} of these tools that makes the framework powerful and flexible. 
The goal of this paper is to provide a simple recipe for future analysis on a wide-range of learning algorithms.
We instantiate the framework on two more examples and either improve or match the state-of-the-art bounds. In the following, $T$ and $\delta$ stand for the number of iterations and the failure probability respectively.

\paragraph{A textbook example: SGD for strongly convex functions.}
We start with a well-known textbook learning algorithm: \textit{stochastic gradient descent (SGD)} algorithm for (smooth) strongly convex functions. We first give an expository proof that matches the state-of-the-art $O(T^{-1}(\log\delta^{-1}+\log\log T))$ convergence rate~\cite{hazan2014beyond,rakhlin2012making} for \textit{strong uniform convergence}. To show the flexibility of the framework, we further use the same analysis with a different stopping time to achieve the optimal $O(T^{-1}\log\delta^{-1})$ convergence rate~\cite{jain2019making,harvey2019tight} for the \textit{last iterate convergence}. The previous analysis for the two convergence guarantees were different but here we can analyze the two under a unifying analysis. See~\autoref{sec:sgd} for details.

\paragraph{A non-convex example: Streaming $k$-PCA.}
Next, we move on to a classic dynamic with a non-convex structure: the \textit{streaming $k$-principle component analysis ($k$-PCA)} problem. Given i.i.d. samples from an unknown distribution, the goal is to use $O(d)$ space and output a vector that is close to the top-$k$ eigenspace of the distribution. For simplicity, we focus on the \textit{gap-dependent and local convergence setting}\footnote{There are other standard guarantees such as global convergence, gap-free convergence, exponential convergence etc. We leave it as an interesting future direction to apply our framework on these other guarantees and simplify/unify the previous proofs.}. The previous state-of-the-art analysis~\cite{AL17}\footnote{\cite{AL17} obtains $\log^5$ in their bound in the global gap-free setting. Using their analysis techniques on the local convergence would get $\log^3$ in the bound.} gets $O(T^{-1}(\log\delta^{-1}+\log^3(dT)))$ where $d$ is the dimension of the problem. We apply our framework and get $O(T^{-1}\log\delta^{-1})$ convergence rate. When $k=1$, this is the first high-probability result that matches the information-theoretic lower bound. See~~\autoref{app:pca} for details. 

\paragraph{An active learning example: Solving stochastic linear bandit with SGD update.}
Finally, we consider a problem with active dynamic where the updates are adaptive and dependant on the whole history: solving stochastic linear bandit with SGD update \cite{jun2017scalable,korda2015fast}. This problem is not only useful for designing scalable bandit algorithms but also serves as an important intermediate step towards analyzing model-free algorithms in linear parameterized Markov decision processes (MDPs)~\cite{jin2018q,jin2019provably}. We analyze an algorithm where previous technique cannot analyze and improve the state-of-the-art regret from $O(d(T\log^2 (T) \log (T/\delta))^{1/2})$ in \cite{jun2017scalable} to $O(d(T\log^2 (T) \log (1/\delta))^{1/2})$. See~\autoref{app:ql} for details.

\subsection{Related work}\label{sec:related work}
We focus on analyzing stochastic processes in learning algorithms which broadly appear in theoretical machine learning~\cite{moitra2018algorithmic,settles2009active,sutton2018reinforcement}, optimization theory~\cite{boyd2004convex,hazan2019introduction,shalev2014understanding}, statistical learning theory~\cite{hastie2009elements}, etc. There have been many beautiful results providing theoretical analysis on a wide range of important learning problems, e.g., principal component analysis~\cite{AL17}, non-negative matrix factorization~\cite{arora2016computing,lee1999learning}, topic models~\cite{arora2013practical,arora2012learning}, matrix completion~\cite{hardt2014understanding,jain2013low}, tensor decomposition~\cite{anandkumar2014tensor,ge2015escaping}, neural networks~\cite{allen2018convergence,du2018gradient,jacot2018neural}, continual learning~\cite{parisi2019continual,kirkpatrick2017overcoming}, etc. However, the lack of a unifying framework often makes the progress in analyzing frontier learning dynamics slow and sub-optimal. This paper attempts to propose a general framework for the future studies in new and complex learning dynamics.

The technical ingredients in the framework are standard, simple, and inspired by the stochastic approximation theory~\cite{Kushner_1997} and a recent analysis for streaming PCA~\cite{CW20}. For example, the stopping time technique or the martingale concentration have been widely applied in many other analysis~\cite{AL17,rakhlin2012making}. We emphasize that it is the \textit{composition} of these tools that makes our framework powerful and flexible, and the main contribution of this paper is to propose a general and streamlined recipe that future analysis of new and complex learning algorithms can easily adopt.

\subsection{Organization of the paper}
The rest of the paper is organized as follows. First, in~\autoref{sec:prelim} we introduce the relevant background knowledge in probability theory. Next, we formally explain our framework in~\autoref{sec:framework} and apply it to the example of SGD. Finally, we compile extra discussion and proofs for the other two problems in the Appendices.

\section{Preliminaries}\label{sec:prelim}
In this section, we formalize the ideas introduced in the previous section by setting up mathematical background on stochastic process in~\autoref{sec:martingale etc} and state the formal version as well as the proof of the pull-out lemma in~\autoref{sec:prelim pull out}.

\paragraph{Notations.}
We use the following notations. $\N=\{1,2,\dots\}$ and $\Nz=\{0,1,\dots\}$. $[T]=\{1,2,\dots,T\}$ and $[T_0:T_1]=\{T_0,\dots,T_1\}$ for every $T\in\N$ and $T_0<T_1\in\Nz$. We use $\{X_t\}_{t\in\Nz}$ to denote the sequence $\{X_0,X_1,\dots\}$ indexed by $\Nz$. We omit the subscript when the context is clear.

\subsection{Stochastic process, concentration inequality, and stopping time}\label{sec:martingale etc}
Stochastic processes are central objects in this paper. Here we introduce preliminary mathematical background in an intuitive language. See~\cite{L16} for a more formal exposition.

A (discrete) \textbf{stochastic process} $\{X_t\}_{t\in\Nz}$ is a sequence of real-valued random variable indexed by $\Nz$. A \textbf{filtration} $\{\cF_t\}_{t\in\Nz}$ for $\{X_t\}$ is a sequence of probability spaces and in this paper we consider $\cF_t$ to be the probability space generated by $X_0,X_1,\dots,X_t$ for every $t\in\Nz$.~\footnote{This can be formally defined via $\sigma$-algebra (see~\autoref{app:tools}). } 
One of the main technical tools for analyzing stochastic processes is the \textbf{concentration inequality}, which states that when one has a good control on the moments of $\{X_t\}$, then the deviation of $\{X_t\}$ is upper bounded with high probability. For example, the following is a variant of the seminal Freedman's inequality~\cite{freedman1975tail}.

\begin{lemma}[A variant of Freedman's inequality]\label{lem:martingale concentration}
Let $\{X_t\}_{t\in\Nz}$ be a stochastic process with filtration $\{\cF_t\}_{t\in\Nz}$. Let $T_0<T\in\Nz$ and $B,\sigma_t,\mu_t\geq0$ be some constants for all $t\in[T]$. Suppose for each $t=T_0+1,T_0+2,\dots,T$, $|X_{t}-X_{t-1}|\leq B$ almost surely, $\Var[X_{t}\, |\, \mathcal{F}_{t-1}]\leq\sigma_t^2$, and $\Exp[X_{t}-X_{t-1}\, |\, \mathcal{F}_{t-1}]\leq\mu_t$, then for every $\delta\in(0,1)$ we have
\[
\Pr\left[\exists t\in[T_0+1:T],\ X_t-X_0\geq 2\max\left\lbrace\sqrt{\sum_{t' = T_0 + 1}^{T_1}\sigma^2_{t'}\log\frac{1}{\delta}}, 2B\log\frac{1}{\delta}\right\rbrace + \sum_{t'=T_0+1}^{T_1}\mu_{t'}\right] < \delta \, .
\]
\end{lemma}
In general, it is difficult to obtain a tight bound on the moment conditions for a concentration inequality. Our framework uses a \textit{stopping time technique} to tackle this issue.
A \textbf{stopping time} $\tau$ is a random variable that takes values in $\Nz$ and is consistent\footnote{Formally, for every $t\in\Nz$, the event $\{\tau=t\}$ should be measurable in the probability space $\cF_t$.} with $\{\cF_t\}$.
For example, let $\{X_t\}$ be an adapted process, the most common stopping time for $\{X_t\}$ is of the form $\tau := \min_{X_t> \Lambda}\{t\}$ for some $\Lambda\in\Real$.
Namely, $\tau$ is the first time when $\{X_t\}$ becomes greater than $\Lambda$. For convenience, we would define the stopping time of this form by ``$\tau$ is the stopping time for $\{X_t> \Lambda\}$''.

Given an adapted process $\{X_t\}$ and a stopping time $\tau$, it is natural to consider the corresponding \textbf{stopped process} $\{X_{t\wedge\tau}\}_{t\in\Nz}$ where $t\wedge\tau=\min\{t,\tau\}$ is also a random variable. That is, $X_{t\wedge\tau}=\mathbf{1}_{\tau\geq t}X_t+\mathbf{1}_{\tau<t}X_{\tau}$ for every $t\in\Nz$. Two useful facts for the future: (i) $X_{t\wedge\tau} - X_{(t-1)\wedge\tau} = \mathbf{1}_{\tau\geq t}(X_t - X_{t-1})$ almost surely for every $t\in\N$. (ii) the stopped process of an adapted process is also adapted and hence concentration inequalities are applicable to a stopped process.

\subsection{The pull-out lemma}\label{sec:prelim pull out}
One of the key tool in our framework is the \textit{pull-out lemma} used in the improvement analysis. Here we restate the lemma and give a whole proof for the completeness of presentation.

\begin{restatable}[The pull-out lemma~\cite{CW20}]{lemma}{pulloutlemma}\label{lem:pull out}
Let $\{M_t\}_{t\in\N}$ be an adapted stochastic process and $\tau$ be a stopping time. For every $T\in\N$, $\Delta\in\Real$, and $\delta\in(0,1)$, suppose
\begin{enumerate}
\item $\Pr[\exists t\in[T],\ M_{t\wedge \tau} > \Delta]< \delta$ and
\item For every $t\in[T]$, $\Pr[\tau \geq t+1\ |\ \exists t'\in[t],\ M_{t'} \leq \Delta] = 1$.
\end{enumerate}
Then, we have
\[
\Pr[\exists t\in[T],\ M_{t} > \Delta]< \delta \, .
\]
\end{restatable}

For the completeness of the presentation, we provide a proof for the pull-out lemma as follows.

\begin{proof}
To simplify the notation, we use $\{M^*_t\}_{t\in\Nz}$ to denote the \textit{maximal process} of $\{M_t\}$ where $M^*_t:=\max_{0\leq t'\leq t}M_{t'}$. Note that the error event $\{\exists t'\in[t],\ M_{t'}>\Delta\}$ now becomes $\{M^*_t>\Delta\}$.

The main idea is to use an auxiliary stopping time $\xi$ for the event $\{M_{t\wedge\tau}^*>\Delta\}$ and decompose the error event $\{M_t^*>\Delta\}$ as follows. 
\begin{align*}
\Pr[M_{t}^*>\Delta]&=\Pr[M_{t}^*>\Delta,\ \tau\geq t]\\&+\Pr[M_{t}^*>\Delta,\ \tau<t,\ \xi\leq\tau]+\Pr[M_{t}^*>\Delta,\ \tau<t,\ \xi>\tau] \, .
\intertext{First, observe that the third term is zero because $\Pr[\tau<t,\ \xi>\tau]=0$. To see this, note that when $\tau<t$, we have $\tau=t\wedge\tau$ and hence $M^*_{t\wedge\tau}=M^*_\tau$. Also, when $\xi>\tau$, we have $M^*_\tau\leq\Delta$ (by the definition of $\xi$). Thus, by the second condition of the lemma statement, we have $\tau\geq\tau+1$, which is a contradiction. So the equation becomes the following.}
&=\Pr[M_{t}^*>\Delta,\ \tau\geq t]+\Pr[M_{t}^*>\Delta,\ \xi\leq\tau<t] \, .
\intertext{Next, observe that when $\tau\geq t$, we have $t=t\wedge\tau$. Also, if $\xi\leq\tau<t$ then $M_{t}^*,M_{t\wedge\tau}^*>\Delta$ according to the definition of $\xi$. Namely, we can turn the process into its stopped process in the above equation as follows.}
&=\Pr[M_{t\wedge\tau}^*>\Delta,\ \tau\geq t]+\Pr[M_{t\wedge\tau}^*>\Delta,\ \xi\leq\tau<t]\\
&\leq\Pr[M_{t\wedge\tau}^*>\Delta]<\delta
\end{align*}
where the last inequality is due to the first condition in the lemma statement. Thus, we have $\Pr[M_{t}^*>\Delta]<\delta$ as desired and complete the proof of the pull-out lemma.
\end{proof}

\section{Our General Framework}\label{sec:framework}
In the theoretical analysis of a learning algorithm, one usually identifies an \textit{objective function} to evaluate how well the algorithm performs. We use $\{X_t\}$ to denote the stochastic process induced by the objective function and the goal is to find the smallest possible rate function $r$ and show that $\{X_t\}$ converges with rate $r$.

\paragraph{User manual.}
Our framework consists of three main steps: the recursion analysis, the moment and concentration analysis, and the improvement analysis. When analyzing a learning algorithm, it often takes few iterations on the three steps to identify the final (and hopefully, optimal) analysis. We also provide a few tips for advanced usage of the framework in~\autoref{sec:framework advanced}.

\paragraph{Prologue: Continuous analysis.}
It is often not obvious how to analyze a discrete stochastic process directly. A general principle inspired by stochastic approximation theory~\cite{Kushner_1997} is to first understand the behaviors of the continuous analog, which is the limiting process by taking the learning rate to $0$. The guidance from the continuous dynamic can usually be very insightful and point to a good way to analyze the discrete stochastic process. See~\cite{CW20} for a nice example and discussion on a stochastic process that is necessary to analyze in two different ways in different regimes and how continuous analysis helps the authors to identify the right recursion to work with.

\subsection{Step 1: Recursion analysis}\label{sec:framewok recursion}
Guided by the continuous analysis, we investigate the local behavior of $\{X_t\}$ by approximating it with a well-studied dynamic. In general, this can be done by designing \textit{recursions} for $\{X_t\}$. For example, the simplest (one-step) recursion could be \textit{linearizing} $\{X_t\}$ as follows.
\[
X_t\leq H_t\cdot X_{t-1}+N_t \, ,\ \forall t\in\N
\]
where $H_t>0$ is a multiplicative factor and $N_t$ is a noise/minor term depending on both $X_{t-1}$ and the stochasticity at the $t$-th step. 

A recursion for $\{X_t\}$ dissects the dominating dynamic of $\{X_t\}$ from the minor stochastic noise. For example, we can unfold the above one-step linearization into the following recursion for $\{X_t\}$.\footnote{This is also known as the \textit{ODE trick} in~\cite{CW20}.}
\begin{equation}\label{eq:recursion}
X_t\leq \left(\prod_{t'=T_0+1}^tH_{t'}\right)\cdot X_{T_0}+M_t \, ,\ \forall\,0\leq T_0< t\in\N
\end{equation}
where $M_{t}=\sum_{t'=T_0+1}^t\prod_{t''=T_0+1}^{t'}H_{t''}^{-1}\cdot N_{t'}$.
Intuitively, $(\prod_{t'=T_0+1}^tH_{t'})\cdot X_{T_0}$ is the \textit{dominating term} that governs the dynamic of $\{X_t\}$ and $M_t$ is the \textit{minor term} that is expected to be relatively small with high probability. Formally, we define a recursion as follows.

\begin{definition}[Recursion]\label{def:recursion}
Let $\{X_t\}_{t\in\N}$ be a stochastic process with a filtration $\{\cF_t\}_{t\in\N}$. A recursion for $\{X_t\}$ is a tuple of adapted processes $(\{D_t\}_{t\in\N},\{M_t\}_{t\in\N})$ with the following properties. (i) $D_0=X_0$ and $M_0=0$. (ii) $X_t\leq D_t+M_t$ almost surely for every $t\in\N$.
\end{definition}

Intuitively, the dominating term $\{D_t\}$ tells us how fast the process $\{X_t\}$ converges and hence the task left is to show that the minor term $\{M_t\}$ is small. In general, a recursion analysis can dissect $\{X_t\}$ into a more complicated form rather than the above linear form $X_t\leq D_t+M_t$. Also, one could use more than one recursions in the analysis and sometimes designing multiple recursions for $\{X_t\}$ would lead to a tighter analysis (e.g.,~\cite{AL17,CW20,harvey2019tight}). In this paper, the framework only focuses on using a single (linear) recursion (i.e.,~\autoref{def:recursion}). The reason is that we already achieve state-of-the-art analysis for all our examples in this simplified setting and the principle in analyzing a single recursion can be easily generalized to multiple recursions. See~\autoref{app:framework} for more discussion on possible extensions.

\paragraph{User manual.}
The recursion step offers a huge design space for the analysis. Once we fix a recursion (or recursions), the next two steps of the analysis will be machinery. Linearization is the easiest way to identify an analyzable recursion for $\{X_t\}$. Especially,~\autoref{eq:recursion} provides a principle\footnote{Note that in a recursion we require the minor process $\{M_t\}$ does not depend on the future, a.k.a., being \textit{adaptive}. In general, this is non-trivial to achieve while linearization guarantees the resulting $\{M_t\}$ is adaptive.} way to dissect $\{X_t\}$ into a drifting term and a minor term. For convenience, we sometimes scale the minor term $\{M_t\}$ by a multiplicative factor in the analysis and that should be clear from the context.

\subsection{Step 2: Moment and concentration analysis}\label{sec:framework moment}
Let $(\{D_t\},\{M_t\})$ be a recursion for $\{X_t\}$. The goal is to show that the minor term $M_t$ is small (with high probability) and hence the dominating term $D_t$ would govern the dynamic of $\{X_t\}$. However, in general $M_t$ could depend on the value of $X_{t'}$ for some $t'<t$ and hence make the analysis challenging. This is the ``chicken and egg'' phenomenon mentioned in~\autoref{sec:intro}.

To resolve the ``chicken and egg'' problem and achieve a tighter analysis, we use a \textit{stopping time} technique to keep track of the \textit{local information on where $\{X_t\}$ is}. Hence, using the moment information of the stopped process of $\{M_{t}\}$, we are able to apply martingale concentration inequality and show that the stopped process of $\{M_t\}$ is dominated by $\{D_t\}$.
We call the collection of such moment bounds a \textit{moment profile} for a recursion $(\{D_t\},\{M_t\})$.

\begin{definition}[Moment profile]\label{def:moment profile}
Let $\{X_t\}_{t\in\Nz}$ be a stochastic process with a filtration $\{\cF_t\}_{t\in\Nz}$ and let $(\{D_t\}_{t\in\Nz},\{M_t\}_{t\in\Nz})$ be a recursion for $\{X_t\}$. Let $\{\Lambda_t\}_{t\in\Nz}$ be a sequence of non-negative thresholds, $\tau$ be the stopping time for the event $\{X_t\geq\Lambda_t\}$, and $\{M_{t\wedge\tau}\}_{t\in\Nz}$ be the stopped process of $\{M_t\}$. A moment profile for the recursion $(\{D_t\},\{M_t\})$ is a tuple of functions $(B,\mu,\sigma^2)$ with the following properties.
\begin{itemize}
\setlength\itemsep{0mm}
\item (Bounded difference) For every $t\geq1$, $|M_{t\wedge\tau}-M_{(t-1)\wedge\tau}|\leq B(t,\{\Lambda_t\})$ almost surely.
\item (Conditional expectation) For every $t\geq1$, $\Exp[M_{t\wedge\tau}-M_{(t-1)\wedge\tau}\ |\ \cF_{t-1}]\leq\mu(t,\{\Lambda_t\})$.
\item (Conditional variance) For every $t\geq1$, $\Var[M_{t\wedge\tau}-M_{(t-1)\wedge\tau}\ |\ \cF_{t-1}]\leq\sigma^2(t,\{\Lambda_t\})$.
\end{itemize}
If we started from time $T_0\in\N$ instead of time $0$, then the moment profile is denoted as $(B_{T_0},\mu_{T_0},\sigma^2_{T_0})$.
\end{definition}

Treating $\{\Lambda_t\}$ as free parameters \textit{isolates} the moment calculation from the potentially complicated global dynamic. Common choices for $\{\Lambda_t\}$ are (i) uniform threshold: $\Lambda_t=\Lambda$ for some $\Lambda>0$ and for all $t\in\N$, and (ii) linear threshold: $\Lambda_t=\Lambda/t$ for some $\Lambda>0$ and for all $t\in\N$. The choice of stopping time $\tau$ is very flexible and hence provides additional room for specialized techniques.

Given a moment profile $(B,\mu,\sigma^2)$ for a recursion $(\{D_t\},\{M_t\})$, we can apply a martingale concentration inequality (e.g.,~\autoref{lem:martingale concentration}) and get the following. For every $\delta\in(0,1)$ and $T\in\N$,
\begin{equation}\label{eq:concentration with stopping time}
\Pr\left[\exists t\in[T],\ M_{t\wedge\tau}-M_0 > \Delta \right] < \delta
\end{equation}
where $\Delta$ is the deviation and is a function of $(B,\mu,\sigma^2,\delta,T)$. In particular, by~\autoref{lem:martingale concentration} we have
\[
\Delta = 2\max\left\lbrace\sqrt{\sum_{t' = 1}^{T}\sigma(t',\{\Lambda_t\})\log\frac{1}{\delta}}, \max_{1\leq t'\leq T}2B(t',\{\Lambda_t\})\log\frac{1}{\delta}\right\rbrace + \sum_{t'=1}^{T}\mu(t',\{\Lambda_t\}) \, .
\]

\paragraph{User manual.}
It is convenient to think of the stopped martingale difference $M_{t\wedge\tau}-M_{(t-1)\wedge\tau}$ as $\mathbf{1}_{\{\tau>t-1\}}\cdot(M_t-M_{t-1})$. To calculate the moment profile, one expresses the three moment quantities of $M_t-M_{t-1}$ as a function of $\{\Lambda_t\}$ conditioning on the event $\{X_{t'}<\Lambda_{t'}\, |\, t'=1,2,\dots,t-1\}$.

\subsection{Step 3: Improvement analysis}\label{sec:framework improvement}
In Step 2, we use a moment profile $(B,\mu,\sigma^2)$ and apply a concentration inequality on the stopped process of the minor term $\{M_t\}$. Ideally, we would like to show that $\{M_t\}$ is small and hence $\{X_t\}$ is dominated by $\{D_t\}$. However, the concentration analysis only works for the stopped process $\{M_{t\wedge\tau}\}$ instead of the original minor process $\{M_t\}$. To resolve this issue, we provide a general and systematic way to \textit{pull out} the stopping time $\tau$ from a concentration inequality via~\autoref{lem:pull out}. 

\pulloutlemma*

Roughly speaking, the pull-out lemma can be proved by partitioning the probability space of the event $\{\exists t\in[T],\ M_t>\Delta\}$ into two parts. In the first part, $\{M_t\}$ agrees with the stopped process $\{M_{t\wedge\tau}\}$ and hence the first condition guarantees that the event happens with a small probability. In the second part, the second condition guarantees that $M_t>\Delta$ would never happen.

After pulling out the stopping time $\tau$ from the concentration inequality, we then have an upper bound on the minor process $\{M_t\}$ and hence can conclude that $\{D_t\}$ dominated the dynamic of $\{X_t\}$. This completes the convergence analysis for $\{X_t\}$.

\paragraph{User manual.}
When instantiating the pull-out lemma to a concrete dynamic, the two conditions in~\autoref{lem:pull out} become some inequalities over $\{\Lambda_t\}$. So the analysis boils down to finding a choice of $\{\Lambda_t\}$ such that (i) all the inequalities are satisfied and (ii) the deviation $\Delta$ is as small as possible.

\subsection{Advanced steps: Flexible analysis}\label{sec:framework advanced}

In the previous three main steps of our framework, we focus on analyzing a single recursion in a single interval. When working on more complicated dynamics, we can extend the basic framework in two ways to achieve a tighter analysis. First, perform an \textit{interval analysis} by dividing the time interval into small pieces and analyzing them separately. This enables tighter analysis in each interval with the cost of a union bound in the end. See~\autoref{sec:sgd} and for an example. Second, design \textit{multiple recursions} that dissect the dominating process better from the minor processes. For example, some previous works~\cite{AL17,CW20,harvey2019tight} use multiple recursions to achieve (nearly) optimal analysis. Nevertheless, the analysis is usually very complicated and specialized to the specific learning dynamic. Our framework, though presented for analyzing a single recursion, provides a general principle to analyze multiple recursions by applying the three-step recipe on each recursion respectively.

\section{Our Analysis of SGD for Strongly Convex Functions}\label{sec:sgd}
In this section, we are going to give a self-contained and unifying proof for two types of high probability convergence guarantees for our running example: SGD for strongly convex functions. Recall that $\{X_t\}$ is the stochastic process that captures how well the SGD algorithm performs and we would like to show that $X_t=O(1/t)$ with high probability. See~\autoref{sec:running example} to review the setup of the problem. We begin by restating the main theorem.

\begin{theorem}[High probability convergence of SGD for strongly convex functions]\label{thm: sgd main appendix}
Let $F$ be a $\lambda$-strongly convex function over some convex domain. An SGD algorithm with gradient bounded by $G>0$ almost surely has the following convergence rate. For every $\delta\in(0,1)$,
\begin{itemize}
\item if $\eta_t=1/(\lambda t)$, for every large enough $T\in\N$, we have the following last iterate convergence
\[
\Pr\left[ X_{T} > \frac{500G^2\log\frac{1}{\delta}}{\lambda^2T} \right] < \delta \, ;
\]
\item if $\eta_t=1/(\lambda t)$, we have the following strong uniform convergence
\[
\Pr\left[\exists  t\in\N,\ X_{t} > \frac{1000 G^2\left(\log\frac{1}{\delta}+2\log\log(t+1)\right)}{\lambda^2 t} \right] < \delta \, .
\]
\end{itemize}
\end{theorem}

\paragraph{Section structure.} In the rest of this section, we analyze a stopped process of $\{X_t\}$ in~\autoref{sec:SGD concentration of stopped} and complete the proof of~\autoref{thm: sgd main appendix} in~\autoref{sec:SGD pull out}. We end the section with a comparison of our proof with the previous work in~\autoref{sec:SGD comparison}.

\subsection{Concentration for the stopped process}\label{sec:SGD concentration of stopped}
Recall that in~\autoref{sec:running chicken and egg} we establish the following recursion for $\{X_t\}$.
\[
X_t\leq\frac{T_0(T_0-1)}{t(t-1)}\cdot(X_{T_0}+M_t)
\]
where $M_{t}=\sum_{t'=T_0+1}^t\frac{t'(t'-1)}{T_0(T_0-1)}\cdot\left(\frac{2(\lambda X_{t'-1}-\hat{\bg}_{t'}^\top(\bw_{t'-1}-\bw^*))}{\lambda t'}+\frac{\|\hat{\bg}_{t'}\|^2}{\lambda^2t'^2}\right)$ for all $t\geq T_0\geq3$. In order to obtain a good bound for $\{M_t\}$, one needs to control its moments well. Specifically, we use a stopping time $\tau$ for the event $\{X_t>\Lambda_t\}$ with parameters $\{\Lambda_t\}$ chosen later to get a tighter control on the stochasticity of $\{M_t\}$.

\begin{lemma}\label{lem:sgd moment details}
Consider the setting in~\autoref{thm: sgd main appendix}. Let $\{\Lambda_t\}$ be a positive sequence and $\tau$ be the stopping time of the event $\lbrace X_t > \Lambda_t \rbrace$. For every $100\leq T_0+1\leq t$, we have the following.
\begin{itemize}[leftmargin=*]
\item (Bounded difference) $|M_{t\wedge\tau}-M_{(t-1)\wedge\tau}|\leq \frac{15G^2t}{\lambda^2 T_0^2}$ almost surely.
\item (Conditional expectation) $\Exp\left[M_{t\wedge\tau}-M_{(t-1)\wedge\tau}\ |\ \cF_{t-1}\right]\leq\frac{2G^2}{\lambda^2T_0^2}$.
\item (Conditional variance) $\Var\left[M_{t\wedge\tau}-M_{(t-1)\wedge\tau} \ |\ \cF_{t-1}\right]\leq\frac{G^2t^2}{\lambda^2T_0^4}\cdot\left(50\Lambda_t+\frac{5G^2}{\lambda^2t^2}\right)$.
\end{itemize}
\end{lemma}
\begin{proof}
Recall that $M_{t\wedge\tau}-M_{(t-1)\wedge\tau}=\mathbf{1}_{\tau\geq t}\cdot(M_t-M_{t-1})$.
For bounded difference, we have
\begin{align}
|M_{t\wedge\tau}-M_{(t-1)\wedge\tau}| &= \left| \mathbf{1}_{\tau\geq t}\cdot \frac{t(t-1)}{T_0(T_0-1)}\cdot\left(\frac{2\left(\lambda X_{t-1}-\hat{\bg}_t^\top(\bw_{t-1}-\bw^*)\right)}{\lambda t} + \frac{\|\hat{\bg}_t\|^2}{\lambda^2t^2}\right) \right| \, . \label{eq: sgd moment diff}
\intertext{By the Cauchy-Schwarz inequality and the facts that $\|\hat{\bg}_{t}\|\leq G$ and $X_t \leq \frac{4G^2}{\lambda^2}$, we get}
&\leq \mathbf{1}_{\tau\geq t}\cdot \frac{t(t-1)}{T_0(T_0-1)} \cdot \left(\frac{2\lambda X_{t-1} + 2G\sqrt{X_{t-1}}}{\lambda t} + \frac{G^2}{\lambda^2t^2}\right) \,
\leq \frac{15G^2t}{\lambda^2T_0^2} \, . \nonumber
\end{align}
For conditional expectation, we apply the fact that $\Exp[\hat{\bg}_{t}^\top(\bw_{t-1}-\bw^*)\ |\ \cF_{t-1}]\geq\lambda X_{t-1}$ to \autoref{eq: sgd moment diff} and get
\begin{align*}
\Exp\left[M_{t\wedge\tau}-M_{(t-1)\wedge\tau}\ |\ \cF_{t-1}\right]\leq\frac{2G^2}{\lambda^2T_0^2} \, .
\end{align*}
For the conditional variance, we have by~\autoref{eq: sgd moment diff} and the definition of $\tau$,
\begin{align*}
\Var\left[M_{t\wedge\tau}-M_{(t-1)\wedge\tau} \ |\ \cF_{t-1}\right] &\leq \Exp\left[\mathbf{1}_{\tau \geq t}\cdot(M_t-M_{t-1})^2\ |\ \cF_{t-1}\right]\\
&\leq \frac{100}{99}\cdot\frac{t^2}{T_0^4}\cdot\left(\Exp\left[\mathbf{1}_{\tau\geq t}\cdot\frac{8\lambda^2X_{t-1}^2+16G^2X_{t-1}}{\lambda^2}+\frac{4G^4}{\lambda^4t^2} \, |\, \cF_{t-1}\right]\right)\\
&\leq\frac{G^2t^2}{\lambda^2T_0^4}\cdot\left(50\Lambda_t+\frac{5G^2}{\lambda^2t^2}\right)\, .
\end{align*}
This completes the proof of~\autoref{lem:sgd moment details}.
\end{proof}



\subsection{Convergence analysis of SGD}\label{sec:SGD pull out}
Now, we prove the two high-probability convergence guarantees for SGD. We also provide an unifying understanding of the proofs in~\autoref{fig:sgd}.

\paragraph{Last iterate convergence.}

\begin{proof}[Proof of item 1 of~\autoref{thm: sgd main appendix}]

In the analysis for last iterate convergence we consider a stopping time $\tau$ defined for the event $\{X_t>\Lambda_t\}$, where $\Lambda_t = \frac{T_1}{t^2}\Lambda$, for an appropriate $\Lambda$ chosen later. Let $T_0=100$ and let $(B_{T_0},\mu_{T_0},\sigma^2_{T_0})$ be the moment profile obtained from~\autoref{lem:sgd moment details}. By~\autoref{lem:martingale concentration}, we have
\[
\Pr\left[\exists T_0+1\leq t\leq T_1,\ |M_{t\wedge\tau}|>\Delta\right]<\delta\, ,
\]
where the deviation $\Delta$ can be upper bounded as follows.
\begin{align*}
\Delta&\leq2\max\left\{\sqrt{\sum_{t=T_0+1}^{T_1}\frac{G^2t^2(50T_1\Lambda/t^2+5G^2/(\lambda^2t^2))}{\lambda^2T_0^4}\log\frac{1}{\delta}},\ \max_{T_0+1\leq t\leq T_1}\frac{15G^2t}{\lambda^2T_0^2}\log\frac{1}{\delta}\right\}+\sum_{t=T_0+1}^{T_1}\frac{2G^2}{\lambda^2T_0^2}\\
&
\leq2\max\left\{\frac{G^2T_1\log\frac{1}{\delta}}{\lambda^2T_0^2}\sqrt{\frac{50\lambda^2\Lambda}{G^2\log\frac{1}{\delta}}+\frac{5}{T_1\log\frac{1}{\delta}}},\ \frac{15G^2T_1\log\frac{1}{\delta}}{\lambda^2T_0^2}\right\}+\frac{2G^2T_1}{\lambda^2T_0^2} \, .
\intertext{Let $\Lambda=\frac{500G^2\log\frac{1}{\delta}}{\lambda^2}$, we further have}
&\leq\frac{400G^2T_1\log\frac{1}{\delta}}{\lambda^2T_0^2} \, .
\end{align*}

Now, in order to \textit{pull out} the stopping time $\tau$ from the concentration inequality, we verify the pull-out condition stated in~\autoref{lem:pull out}, which, in this setting, is
\[
\Pr\left[X_{t'}<\frac{T_1^2}{t'^2}\Lambda,\ \forall T_0+1\leq t'\leq t\, |\, M_{t'}\leq\Delta,\ \forall T_0+1\leq t'\leq t\right]=1\, ,
\]
for every $T_0+1\leq t\leq T_1$. Observe that when $M_{t'}\leq\Delta$ for all $t'\leq t$, for every $T_0+1\leq t'\leq t$, by the recursion we have 
\begin{align*}
X_{t'}&=\prod_{t''=T_0+1}^{t'}(1-2\eta_{t''}\lambda)(X_{T_0}+M_{t'}) =\frac{T_0(T_0-1)}{t'(t'-1)}(X_{T_0}+M_{t'})\\
&\leq\frac{T_0^2}{t'^2}(X_{T_0}+\Delta) <\frac{T_0^2}{t'^2}\frac{500G^2T_1\log\frac{1}{\delta}}{\lambda^2T_0^2}=\frac{T_1\Lambda}{t'^2} \, .
\end{align*}
Thus, by the pull-out lemma (i.e.,~\autoref{lem:pull out}) we have
\[
\Pr\left[\exists T_0+1\leq t\leq T_1,\ |M_{t}|>\frac{500G^2T_1\log\frac{1}{\delta}}{\lambda^2T_0^2}\right]<\delta \, .
\]
Combine with the recursion and the fact that $X_{T_0}\leq\frac{4G^2}{\lambda^2}$ due to the $G$-boundedness of $F$, we have
\[
\Pr\left[X_{T_1}>\frac{500G^2\log\frac{1}{\delta}}{\lambda^2T_1}\right]<\delta \, .
\]
This completes the proof for the last iterate convergence of SGD.
\end{proof}

\paragraph{Strong uniform convergence.}
It turns out that to achieve tighter convergence rate for strong uniform convergence, we have to apply concentration analysis on different intervals and apply an union bound in the end. Here, we modularize the concentration analysis for a single interval into the following lemma.

\begin{lemma}\label{lem:sgd concentration details}
Consider the setting in~\autoref{thm: sgd main appendix}. For every $100\leq T_0\in\N$, let  $(B_{T_0},\mu_{T_0},\sigma^2_{T_0})$ be the moment profile for $\{X_t\}$ and $T_0$ from~\autoref{lem:sgd moment details}. If $\eta_t=1/\lambda t$, then for every $A_0>A_1>0$, $\delta'>0$, $\Lambda>0$, and $100\leq T_0<T_1\in\N$, let
\[
\Delta=\frac{G^2T_1\log\frac{1}{\delta'}}{\lambda^2T_0^2} \left(\sqrt{\frac{70T_1\Lambda\lambda^2}{G^2
\log\frac{1}{\delta'}}} + 50\right) \, .
\]
Suppose that we have $A_0 + \Delta < \Lambda$ and $\frac{T_0(T_0-1)}{T_1(T_1-1)} \Lambda < A_1$. Then, 
\[
\Pr\left[\exists{T_0 + 1\leq t\leq T_1}\, ,\ X_t > A_1 \ \middle| \ X_{T_0} < A_0\right] < \delta' \, .
\]
\end{lemma}
\begin{proof}
Let $\tau$ be the stopping time for the event $\{X_t>\Lambda\}$.
Given the moment profile for $\{X_t\}$ in~\autoref{lem:sgd moment details}, we apply~\autoref{lem:martingale concentration} and get
\[
\Pr\left[\exists T_0+1\leq t\leq T_1,\ M_{t\wedge\tau}<\Delta \right] < \delta'
\]
where
\begin{align*}
\Delta\leq\ &\max\left\{2\sqrt{\sum_{t' = T_0 + 1}^{T_1}\frac{G^2t'^2}{\lambda^2T_0^4}\left(50\Lambda+\frac{5G^2}{\lambda^2t'^2}\right)\log\frac{1}{\delta'}} , 30\frac{G^2T_1}{\lambda^2T_0^2}\log\frac{1}{\delta'}\right\} + \sum_{t'=T_0+1}^{T_1}\frac{2G^2}{\lambda^2 T_0^2}\\
\leq\ & \sqrt{\frac{70G^2T_1^3\Lambda\log\frac{1}{\delta'}}{\lambda^2T_0^4}} + \frac{G^2\sqrt{20T_1\log\frac{1}{\delta'}}}{\lambda^2T_0^2} +  30\frac{G^2T_1}{\lambda^2T_0^2}\log\frac{1}{\delta'} + \frac{2G^2T_1}{\lambda^2 T_0^2}\\
\leq\ &\frac{G^2T_1\log\frac{1}{\delta'}}{\lambda^2T_0^2} \left(\sqrt{\frac{70T_1\Lambda\lambda^2}{G^2
\log\frac{1}{\delta'}}} + 50\right) \, .
\end{align*}
By~\autoref{eq:sgd recursion}, for every $T_0+1\leq t\leq T_1$ we have
\[
X_{t}\leq\frac{T_0(T_0-1)}{t(t-1)}\cdot\left(X_{T_0}+M_t\right)
\]
and hence $A_0+\Delta<\Lambda$ implies the pull-out condition of~\autoref{lem:pull out}. Thus, we conclude that $\Pr[\exists{T_0 + 1\leq t\leq T_1}\, ,\ X_t > A_1 \, |\, X_{T_0} < A_0] < \delta'$ as desired.
\end{proof}

We now prove the strong uniform convergence bound for SGD.
\begin{proof}[Proof of item 2 in~\autoref{thm: sgd main appendix}]
Let $t_0=0$, $t_1=100$, $a_0=\frac{1000G^2\log\frac{1}{\delta}}{\lambda^2t_1}$. For each $i\geq2$, let $t_i=2t_{i-1}$. For each $i\geq1$, let
\[
\delta_i=\frac{\delta}{2i^2},\ a_i = \frac{1000 G^2\log\frac{1}{\delta_i}}{\lambda^2 t_i}\text{, and } \Lambda_i = 2a_{i-1} \, .
\]
Now, for each $i\in\N$, we invoke~\autoref{lem:sgd concentration details} with $A_0=a_{i-1}$, $A_1=a_i$, $T_0=t_{i-1}$, $T_1=t_i$, and $\delta'=\delta_i$.
Let us verify the two conditions. First, we verify the pull out condition as follows.
\begin{align*}
a_{i-1} + \Delta_i
&=a_{i-1} + \frac{G^2t_i\log\frac{1}{\delta_i}}{\lambda^2t_{i-1}^2} \left(\sqrt{\frac{70t_{i-1}\Lambda_i\lambda^2}{G^2
\log\frac{1}{\delta_i}}} + 50\right)\\
&= a_{i-1} + \frac{2\sqrt{70\cdot 2000}\cdot a_{i-1}}{1000} + \frac{200a_{i-1}}{1000} < 2a_{i-1} = \Lambda_i \, .
\end{align*}
Next, we verify the improvement condition $\prod_{t=t_{i-1}+1}^{t_i}(1-2\eta_t\lambda)(a_{i-1}+\Delta_i) < a_{i}$ as follows.
\[
 \frac{t_{i-1}(t_{i-1} - 1)}{t_i(t_i - 1)} \cdot 2a_{i-1} < \frac{2a_{i-1}}{4} \leq a_i \, .
\]
As $a_i\geq \frac{1000G^2(\log(1/\delta)+\log\log (t_{i-1}+1))}{\lambda^2t_{i}}$ for every $t_{i-1}+1\leq t\leq t_i$, by~\autoref{lem:sgd concentration details}, this implies that
\[
\Pr\left[\exists t_{i-1} + 1\leq t\leq t_i \, ,\ X_{t} > \frac{1000 G^2\left(\log\frac{1}{\delta}+\log\log (t+1)\right)}{\lambda^2 t} \, \middle|\, X_{t_{i-1}}\leq a_{i-1} \right] < \delta_i \, .
\]
Also, since $\sum_{i=1}^\infty\delta_i\leq\delta$ and $X_{t}\leq \frac{4G^2}{\lambda^2}$ for all $t\leq t_1$ due to the $G$-boundedness, by union bound we have
\[
\Pr\left[\exists  t\in\N,\ X_{t} > \frac{1000 G^2\left(\log\frac{1}{\delta}+\log\log (t+1)\right)}{\lambda^2 t} \right] < \delta \, .
\]
This completes the proof for the strong uniform convergence of SGD.
\end{proof}

\begin{figure}[ht]
    \centering
    \includegraphics[width=15cm]{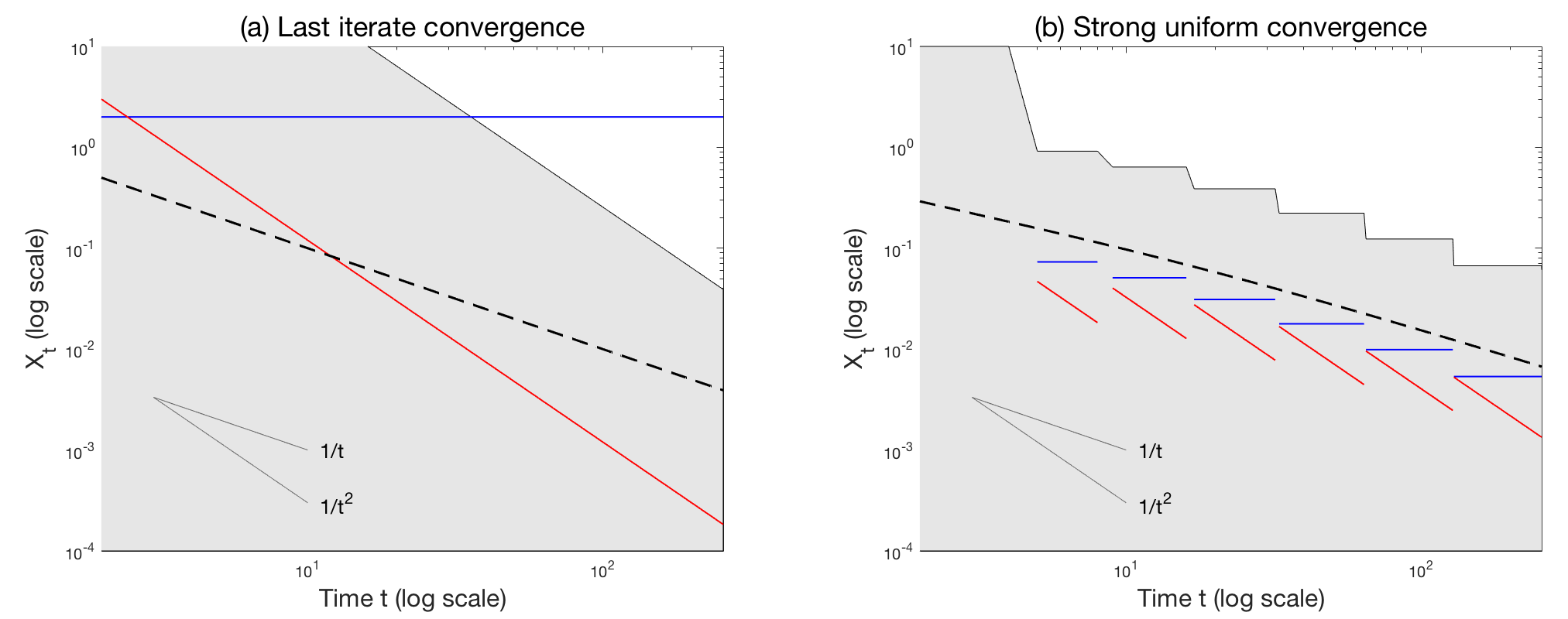}
    \caption{\textbf{A unifying analysis for the convergence of SGD.} The $x$-axis is for the time (in log scale) and the $y$-axis is for the value of $X_t$ (in log scale). The gray area denotes the regime where $\{X_t\leq\Lambda_t\}$. The {\color{blue} blue} line represents the deviation bound $\{\Delta_t\}$ for $\{M_{t\wedge\tau}\}$ while the {\color{red} red} line represents the upper bound for $\{X_{t\wedge\tau}\}$. The pull-out condition is equivalent to ``the red line lying within the gray area''. The dashed line represents the desired convergence speed $1/t$ and $\log\log(t)/t$ respectively. On the other hand, the decreasing speed of the red line is $O(1/t^2)$. In the last iterate convergence (\autoref{fig:sgd}~(a)), the goal is to show that the red line is below the dashed line \textit{at the last time point},. Thus, we allow the beginning of the red line to be higher and can analyze the process in a single interval. As for the strong uniform convergence (\autoref{fig:sgd}~(b)), the goal is to show that the red line is \textit{always} below the dashed line. This is impossible to achieve in a single interval because the red line is steeper than the dashed line. Nevertheless, we overcome this via an interval analysis by paying an extra $\log\log(t)$ union bound factor.}
    \label{fig:sgd}
\end{figure}

\subsection{Comparison with previous analysis}\label{sec:SGD comparison}
The last iterate convergence and the strong uniform convergence of SGD with the same rates of our~\autoref{thm: sgd main appendix} had been proved in~\cite[Theorem~7.5]{harvey2019tight} and~\cite[Proposition~1]{rakhlin2012making} respectively. Conceptually, our proof and techniques improve their analysis in two ways. First, the analysis of~\cite{harvey2019tight} and~\cite{rakhlin2012making} used different recursions for $\{X_t\}$. Namely, to show different types of convergence guarantees, the previous work had to deal with different noise processes while our proof provides a unifying understanding (see~\autoref{fig:sgd}). Second, both~\cite{harvey2019tight} and~\cite{rakhlin2012making} required specialized techniques (i.e., Lemma~4.1 of~\cite{harvey2019tight} and Lemma~3 of~\cite{rakhlin2012making}) in the heart of their proofs. Concretely, to tackle the chicken and egg problem, they had to prove new concentration inequalities specialized to the recursion they were using. On the contrary, one does not need to prove any new concentration inequality when using our framework. The ``stopping time + pull-out lemma'' package reduces the chicken and egg problem to properly designing a stopping time and verifying the pull-out condition. See~\autoref{sec:framework} for more detailed explanations and user manuals of our framework.

Finally, we want to highlight that although we do not improve the state-of-the-art\footnote{In fact, we believe the current rate is information-theoretically optimal.} convergence rate of SGD, we use our framework to improve the convergence rate analysis for other problems such as streaming PCA (see~\autoref{app:pca}). It is an interesting future direction to systematically apply our framework to other learning algorithms and get a simpler and more unifying proof as well as improve the convergence rate.

\paragraph*{Acknowledgement}
We thank Kai-Min Chung for useful discussion in the early stage of this work and thank Boaz Barak and Suvrit Sra for helpful comments on a draft of this paper. We also thank Madhu Sudan for telling us the puzzle in the introduction.
CC is supported by NSF awards CCF 1565264 and CNS 1618026. JS is supported by DARPA ONISQ program award HR001120C0068. MW is supported by NSF Awards CCF-1810758, CCF-0939370 and CCF-1461559. TY is supported by NSF BIGDATA grant 1741341.

\bibliography{mybib}
\bibliographystyle{alpha}

\newpage
\appendix
\begin{center}
    \Huge\textsc{Appendix}
\end{center}
\begin{itemize}
\setlength\itemsep{0mm}
    \item \autoref{app:tools} provides sufficient tools and mathematical backgrounds.
    \item \autoref{app:puzzle} provides details on the puzzle in the introduction.
    \item \autoref{app:framework} provides an in-depth discussion on the framework.
\item \autoref{app:pca} provides the details on the example of local convergence of $k$-PCA.
    \item \autoref{app:ql} provides the details on the example of linear bandit with SGD updates.
\end{itemize}

\section{Tools and Preliminaries}\label{app:tools}

\subsection{Common convergence guarantees for learning dynamics}\label{sec:convergence types}
Let $\{X_t\}$ be a stochastic process. In this paper, we focus on the case where $\{X_t\}$ takes non-negative values and the goal is to show that $X_t\rightarrow0$ as $t\rightarrow\infty$. Here, we define two common variants of high probability convergence guarantee.

\begin{definition}[Last iterate convergence]\label{def:convergence last iterate}
Let $\delta\in(0,1)$ and $r:\N\times(0,1)\rightarrow\Real$ be a rate function. For every stochastic process $\{X_t\}$, we say $\{X_t\}$ has last iterate convergence with rate $r$ if $\Pr[X_t>r(t,\delta)]<\delta$ for every $t\in\N$.
\end{definition}

\begin{definition}[Strong uniform convergence]\label{def:convergence strong uniform}
Let $\delta\in(0,1)$ and $r:\N\times(0,1)\rightarrow\Real$ be a rate function. For every stochastic process $\{X_t\}$, we say $\{X_t\}$ has strong uniform convergence with rate $r$ if $\Pr[\exists t\in\N,\ X_t>r(t,\delta)]<\delta$.
\end{definition}

\subsection{Matrix norms and inequalities}
As many common potential functions are defined as the norm of certain matrix, here we provide some common matrix norms and inequalities that will be useful. 

\begin{definition}[Matrix norms]\label{def:matrix norms}
Let $A\in\Real^{n\times m}$.
\begin{itemize}
\item The Frobenius norm of $A$ is defined as
\[
\|A\|_F:=\sqrt{\tr(A^\top A)} \, .
\]
\item The operator norm of $A$ is defined as
\[
\|A\|:=\sup_{\bx\in\Real^n\backslash\{0^n\}}\frac{\|A\bx\|_2}{\|\bx\|_2} \, .
\]
\item The Schatten $p$ norm of $A$ for some $p\geq1$ is defined as
\[
\|A\|_p:=\tr(|A|^p)^{1/p}
\]
where $|A|:=\sqrt{A^\top A}$.
\item The matrix inner product of $\bx\in\Real^n$ and $\mathbf{y}\in\Real^m$ is defined as
\[
\langle\bx,\mathbf{y}\rangle_A:=\bx^\top A\mathbf{y} \, .
\]
When $A$ is a square matrix, the $A$-norm of $\bx$ is defined as
\[
\|\bx\|_A:=\sqrt{\langle\bx,\bx\rangle_A} \, .
\]
\end{itemize}
\end{definition}

Note that $\|A\|_F=\|A\|_2$ and $\|A\|_p^p=\sum_{i=1}^{n\wedge m}\sigma_i(A)^p$ where $\sigma_1(A)\geq\cdots\geq\sigma_n(A)\geq0$ are the singular values of $A$. Therefore, we have $\|AA^\top\|_1 = \|A\|_F^2$ and $\|A\|_\infty = \|A\|$.

\begin{lemma}[Matrix inequalities]\label{lem:matrix inequalities}
Let $A\in\Real^{n\times m}$ and $B\in\Real^{m\times k}$.
\begin{itemize}
\item $\|AB\|_F\leq\|A\|\|B\|_F$.
\item $\|A\|\leq\|A\|_F\leq\sqrt{m}\|A\|$.
\item (Matrix Cauchy-Schwarz inequality): $|\tr(AB)|\leq\|A\|_F\|B\|_F$.
\item (Matrix H\"{o}lder inequality): $|\tr(AB)|\leq\|A\|_p\|B\|_q$ for every $p,q\geq1$ such that $1/p+1/q=1$.
\item (Matrix AM-GM inequality): $|2\tr(AB)|\leq \tr(A^\top A) + \tr(B^\top B)=\|A\|_F^2+\|B\|_F^2$.
\end{itemize}
\end{lemma}

\section{Details on the Puzzle about Conditional Expectation}\label{app:puzzle}
Let us first calculate the expected number of times to roll a dice before getting the number $1$. Let $T$ be this number and we have
\begin{equation}\label{eq:puzzle 1}
\Exp[T] = \sum_{t=1}^\infty t\cdot\Pr[T=t] = \sum_{t=1}^\infty t\cdot\left(1-\frac{1}{6}\right)^{t-1}\frac{1}{6} = 6 \, .
\end{equation}
Now, let us calculate the expectation of $X$ again conditioned on the dice only outputs an odd number. Let $E$ be the event of the dice only outputting an odd number and repeat the calculation as follows.
\begin{equation}\label{eq:puzzle 2}
\Exp[T\, |\, E] = \sum_{t=1}^\infty t\cdot\Pr[T=t\, |\, E] = \sum_{t=1}^\infty t\cdot\left(1-\frac{1}{3}\right)^{t-1}\frac{1}{3} = 3 \, .
\end{equation}
But we claimed in the introduction that the right answer should be $3/2$, what's wrong with the above calculation!?
It turns out that there are two ways to interpret ``conditioning on the event that you have only seen odd numbers'': (i) the dice never outputs an even number or (ii) we terminate the throwing process and redo it if there's an even number showing up. For (i),~\autoref{eq:puzzle 2} faithfully captures the scenario and does the calculation. However, apparently the question is asking about interpretation (ii). 

So what would be the right way to analyze this conditional expectation? Notice that stopping time is the right formalism for modeling this probability event. Namely, let $\{X_t\}_{t\in\N}$ be the outcomes of the dice and define $\tau:=\min_{X_t=1,\forall t'<t,\ X_{t'}\neq1}\{t\}$ and $\xi:=\min_{X_t\text{ even}, \forall t'<t,\ X_{t'}\text{ odd}}\{t\}$. Now we can rewrite the conditional expectation as $\Exp[\tau\, |\, \tau<\xi]$. Furthermore, note that
\[
\Exp[\tau\, |\, \tau<\xi] \overset{\text{(i)}}{=} \Exp[\tau\wedge\xi\, |\, \tau<\xi] \overset{\text{(ii)}}{=} \Exp[\tau\wedge\xi] \overset{\text{(iii)}}{=} \frac{3}{2}
\]
where (i) holds because in the conditional probability space $\tau<\xi$ we have $\tau=\tau\wedge\xi$ with probability $1$; (ii) holds by the symmetry of the dice\footnote{Concretely, let $\tau_i:=\min_{X_t=i,\forall t'<t,\ X_{t'}\neq1}\{t\}$ for each $i\in[6]$. Notice that by symmetry we have $\Exp[\tau\wedge\xi\, |\, \tau=\tau\wedge\xi]=\Exp[\tau\wedge\xi\, |\, \tau_i=\tau\wedge\xi]$ for all $i\in\{1,2,4,6\}$. As the events $\{\tau_i=\tau\wedge\xi\}_{i\in\{1,2,4,6\}}$ partition the whole probability space, we have $\Exp[\tau\wedge\xi\, |\, \tau<\xi]=\Exp[\tau\wedge\xi]$ as desired.}; (iii) holds due to the same calculation of~\autoref{eq:puzzle 1}. To sum up, the usage of stopping time not only captures the right conditional probability space but also helps identifying symmetry and avoiding complications.


\section{Details on the Framework}\label{app:framework}
In this section, we elaborate on the framework and provide complete proofs for the lemmas and propositions in the main article.

\subsection{Why non-linearity in the dynamic creates entanglement?}
In this subsection, we talk about how non-linearity in the stochastic updates can create the entanglement between the noise and the process. Consider the following general update rule
\[
X_t = f(X_{t-1}, N_t).
\]
By doing Taylor expansion on $f$, we have
\[
X_t = a_0(N_t) + a_1(N_t)X_{t-1} + a_2(N_t)X_{t-1}^2 + \dotsb \, . 
\]
For $f$ to be nonlinear, either for $i\geq 1$, $a_i(N_t)$ is not a constant or for $i>1$, $a_i(N_t)\neq 0$. If $a_i(N_t)$ is not constant, then the update rule is entangled already by looking at $a_i(N_t)X_{t-1}^i$. So it suffices to assume that only $a_0(N_t)$ depends on $N_t$. Now if $i > 1$ and $a_i(N_t)\neq 0$, we can unfold the expression one more step and get
\[
a_iX_{t-1}^i = a_i(a_0(N_{t-1}) + a_1(N_{t-1})X_{t-2} + a_2(N_{t-1})X_{t-2}^2 + \dotsb)^i = ia_ia_0(N_{t-1})(a_iX_{t-2}^i)^i + \dotsb \, .
\]
We can see that $a_0(N_{t-1})$ and $X_{t-2}$ are entangled together. 
\subsection{Continuous analysis}
In this subsection, we talk about how continuous analysis can serve as a guide on how to analyze the discrete dynamic. This subsection is not needed for the use of the framework, so the reader is welcome to skip it for the first time reading through this paper. The continuous analysis helps us in three ways: 
\begin{enumerate}
    \item Determine the intrinsic behaviors of the dynamic.
    \item Give us a guide on how to analyze the dynamic.
    \item Help us to write down a closed-form solution of the dynamic that is adapted.
\end{enumerate}

\paragraph{Determine the intrinsic behaviors of the dynamic.} One important thing to do is to understand the intrinsic behaviors of the dynamic first. For example, if at the continuous limit, the system is chaotic, it is probably useless to analyze further since we can't hope the discrete dynamic to be any better than chaos. Therefore, one natural way to study the stochastic system is to consider its continuous limit and then study the corresponding random dynamical system to characterize different fixed points, saddle points, limit cycles, etc. Only by fully understanding the continuous counterpart, we can cope with the intrinsic behaviors of the dynamic in the discrete setting. This leads us to our second point.

\paragraph{Give us a guide on how to analyze the dynamic.} By looking at the continuous system as a random dynamical system, we can obtain a strategy to analyze the dynamic. For example, even without the noise, it is not obvious how to analyze
\[
X_t = X_{t-1} + \eta X_{t-1}(1-X_{t-1}) \, .
\]
However, at the continuous limit, this gives us
\[
dX_t =  X_{t-1}(1-X_{t-1})dt
\]
which has a stable fixed point at $1$ and an unstable fixed point at $0$. This suggests that around $0$, we should linearize at $0$
\[
X_t = (1 + \eta(1-X_{t-1}))X_{t-1} \, .
\]
On the other hand, when the dynamic gets closer to $1$, we should linearize at $1$,
\[
X_t - 1 = (1 - \eta X_{t-1})(X_{t-1}-1) \, .
\]
We recommend~\cite{CW20} for a detailed discussion on related ideas.

\paragraph{Help us to write down a closed-form solution of the dynamic that is adapted.} In a stochastic differential equation, we solve the dynamic by writing it down as a stochastic integral. This suggests that to solve a stochastic difference equation, we should write it as a linear combination of the noise. This is how~\autoref{eq:recursion} appears. Recall that given a linear dynamic
\[
X_t = H_t\cdot X_{t-1} + N_t \, ,
\]
\autoref{eq:recursion} gives us
\[
X_t = \prod_{i=1}^tH_t \left(X_0 + \sum_{i=1}^t\frac{N_i}{\prod_{j=1}^iH_j}\right) \, .
\]
Comparing with the continuous counterpart, we have
\[
\frac{dX(t)}{dt} = H(t)X(t) + N(t)
\]
and
\[
X(t) = e^{H(t)}\left(X(0) + \int_{0}^te^{-H(s)}N(s)ds\right) 
\]
which we can see the correspondence easily. Furthermore, some readers might argue that~\autoref{eq:recursion} is simply an unrolling recursion to write down a closed-form solution, but we claim that it is a very special closed-form solution in which the summation of the noise term is adapted. 

Notice that $\sum_{i=1}^t\prod_{j=1}^iH_j^{-1}N_i$ in~\autoref{eq:recursion} does not depend on the future event. This is naturally true because it is the discrete counterpart of a stochastic integral. However, if we write the closed-form solution as
\[
X_t = \prod_{i=1}^tH_t X_0 + \sum_{i=1}^t\prod_{j=i+1}^tH_jN_i \, ,
\]
the process $\sum_{i=1}^t\prod_{j=i+1}^tH_jN_i$ is no longer adapted and hence we are not able to apply martingale concentration technique. In the situation where we need to write down closed form solution not from a linear approximation, translating how the continuous counterpart writes down the stochastic integral will help us to write down a closed form solution with adapted noise.

\subsection{Recursion analysis}
In this section, we provide the proof of \textit{unfolding the recursion}, \textit{a.k.a.,} the ODE trick in~\cite{CW20}.
\begin{lemma}\label{lem:ODE trick}
Let $\{X_t\}_{t\geq\Nz}$, $\{N_t\}_{t\in\N}$, and $\{H_t\}_{t\in\N}$ be sequences of random variables with the following dynamic
\begin{equation}\label{eq:ODE trick 1}
X_{t} \leq H_t\cdot X_{t-1}+N_t
\end{equation}
for all $t\in\N$. Then for all $T_0,t\in\Nz$ such that $T_0<t$, we have
\[
X_t \leq D_t\cdot(X_{T_0}+M_t)
\]
where $D_t=\prod_{t'=T_0+1}^tH_{t'}$ and $M_t=\sum_{t'=T_0+1}^tD_{t'}^{-1}\cdot N_{t'}$.
\end{lemma}
\begin{proof}
For each $T_0+1\leq t'\leq t$, dividing~\autoref{eq:ODE trick 1} with $D_{t'}$ on both sides, we have
\[
\frac{X_{t'}}{D_{t'}} = \frac{X_{t'-1}}{D_{t'-1}} + \frac{N_{t'}}{D_{t'}} \, .
\]
We get the desiring expression by telescoping the above equation from $t'=T_0+1$ to $t$.
\end{proof}

\subsection{Improvement analysis}

Here, we instantiate the improvement analysis into the following proposition that could be convenient to apply in future analysis.

\begin{restatable}[Improvement analysis]{proposition}{concentration}\label{prop:concentration}
Let $\{X_t\}$ be a stochastic process described in~\autoref{eq:recursion} for some $\{D_t\}$ and $\{M_t\}$. Let $\Lambda>0$, $\tau$ be the stopping time for the event $\{X_t>\Lambda\}$, and $(B,\mu,\sigma^2)$ be a moment profile for $\{X_t\}$ and $\Lambda$. For every $T_1\in\N$ and $\delta'>0$, let
\[
\Delta=\Delta\left(B,\mu,\sigma^2,\Lambda,T_1,\delta'\right)
\]
be the deviation from a concentration inequality\footnote{For example, $\Delta=\sqrt{2\sum_{t' = T_0 + 1}^{T_1}B_{T_0}^2(t', \Lambda)\log(1/\delta')}+ \sum_{t'=T_0+1}^{T_1}\mu_{T_0}(t', \Lambda)$ if we used Azuma's inequality. See~\autoref{app:framework} and~\cite{Fan2006} for more examples.} such that $\Pr\left[\max_{1\leq t\leq T_1}M_{t\wedge\tau} > \Delta\right]<\delta' \, .$ \\If we have 
\begin{itemize}
\item (Improvement condition) $D_{T_1}\cdot(A_0+\Delta)\leq A_1$ and
\item (Pull-out condition) $D_t\cdot(A_0+\Delta)\leq\Lambda$ for every $1\leq t\leq T_1$.
\end{itemize}
Then we have 
\[\Pr\left[\exists t\in [T_1],\ X_t > D_t\cdot(A_0+\Delta)\ \middle|\ X_{0}\leq A_0\right]<\delta' \, .
\]
In particular, the above implies $\Pr[X_{T_1}>A_1\ |\ X_{0}\leq A_0]<\delta'$. Also, the proposition can naturally extend to starting from $T_0\in\N$ instead of $0$.
\end{restatable}

\begin{proof}
We would like to apply the pull-out lemma, \textit{i.e.,}~\autoref{lem:pull out}, on $\{M_{t\wedge\tau}\}_{t\in\N}$ and thus have to verify the following two conditions. First, note that $\{M_{t\wedge\tau}\}_{t\in\N}$ forms a martingale. Thus, due to the martingale concentration inequality and the choice of $\Delta$, we have
\[
\Pr\left[\max_{1\leq t\leq T_1}M_{t\wedge\tau}>\Delta\right]<\delta'
\]
and thus we satisfy the first condition of the pull-out lemma. Next, for every $1\leq t\leq T_1$, suppose $\max_{1\leq t'\leq t}M_{t'}\leq\Delta$, then we have the following from the recursion formula (see~\autoref{eq:recursion}).
\[
X_{t}=D_t\cdot\left(X_{0}+M_t\right)\leq D_t\cdot(A_0+\Delta)\leq\Lambda
\]
where the last inequality is from the pull-out condition in the proposition statement. Thus, by the choice of $\tau$, we have
\[
\Pr\left[\tau>t\ \Bigg|\ \max_{1\leq t'\leq t}M_{t'}\leq\Delta\right]=1 \, . 
\]
The above two satisfy the second condition of the pull-out lemma as desired. As a result, by the pull-out lemma (see~\autoref{lem:pull out}), we have pulled out the stopping time as follows.
\[
\Pr\left[\max_{1\leq t\leq T_1}M_t>\Delta\right]<\delta' \, .
\]
Finally, by the recursion formula, we have
\[\Pr\left[\exists t\in [T_1],\ X_t > D_t\cdot(A_0+\Delta)\ \middle|\ X_{0}\leq A_0\right]<\delta' \, .
\]
In particular, by combining with the improvement condition in the proposition statement, we have
\[
\Pr[X_{T_1}>A_1\ |\ X_{0}\leq A_0]<\delta'
\]
as desired.
\end{proof}

\subsection{Interval analysis}

Now, we are going to see how to systematically implement a tight interval analysis using~\autoref{prop:concentration}. Recall that the proposition gives the improvement guarantee $\Pr[X_{T_1}>A_1\ |\ X_{T_0}\leq A_0]<\delta'$ as long as the two conditions are satisfied. Recall that we start from $X_0$ and want to see how small $X_T$ could be with high probability. As the first try, we can invoke~\autoref{prop:concentration} by setting $T_0=0$, $T_1=T$, and $A_0=X_0$ then see what is the smallest $A_1$ we can get. Nevertheless, such analysis in general would not be tight because it does not use the local information.

To fully leverage the improvement analysis, we perform an interval analysis by designing sequences $\{a_i\}$, $\{t_i\}$, $\{\delta_i\}$ of length $\ell$ (or $\ell+1$) such that we invoke~\autoref{prop:concentration} by setting $T_0=t_{i-1}$, $T_1=t_i$, $A_0=a_{i-1}$, $A_1=a_i$, and $\delta'=\delta_i$ for each $i=1,2,\dots,\ell$. Namely, in the $i$-th interval, we would like to show $\Pr[X_{t_i}>A_i\ |\ X_{t_{i-1}}\leq A_{i-1}]<\delta_i$ and thus by union bound we would have $\Pr[X_{t_\ell}>A_\ell\ |\ X_{t_0}\leq A_0]<\sum_i\delta_i$.

Note that with the general recipe as above, in principle one can get the tightest bound by solving the following optimization problem.
\begin{center}
    \fbox{%
    \parbox{0.65\textwidth}{%
	\begin{align*}
		& \underset{\ell\in\N,\{t_i\},\{a_i\},\{\delta_i\},\{\Lambda_i\}}{\text{minimize}}
		& & a_\ell \\
		& \text{subject to}
		& & a_i\geq\prod_{t=t_{i-1}+1}^{t_i}H_t\cdot(a_{i-1}+\Delta_i) \, ,\ \forall i=1,2,\dots,\ell\\
		& & & a_{i-1}+\Delta_i\leq\Lambda_i \, ,\ \forall i=1,2,\dots,\ell\\
		& & & \Pr[X_0>a_0]<\delta_0 \\
		& & & \sum_{i}\delta_i\leq\delta \\
		& & & 0=t_0<t_1<\cdots<t_\ell=T \, .
	\end{align*}
    }%
}
\end{center}
However, in general the above optimization problem might be complicated to solve optimally by hands. Thus, we provide some common ways to set the intervals as a principle to implement interval analysis.

\paragraph{How to set $\{a_i\}$.}
For simplicity, let us focus on the setting where the goal is moving from $X_0$ to $\epsilon$ where $X_0\gg\epsilon>0$. Namely, $a_0=X_0$ and $a_\ell=\epsilon$. The principle here can be easily adapted to other settings.
We provide three common ways of setting $\{a_i\}$: the \textit{greedy improvement}, the \textit{multiplicative improvement}, and the \textit{polynomial improvement}. See~\autoref{table:interval a} for a summary. 
\begin{table}[ht]
\centering
\def\arraystretch{2}
\begin{tabular}{|c|c|c|}
\hline
Type           & $\{a_i\}$                                                                    & \# intervals                                                \\ \hline
Greedy       & Pick $a_i$ as small as possible                                                  & Problem-dependent                                             \\ \hline
Multiplicative & $a_i=\frac{a_{i-1}}{2}=2^{-i}\cdot X_0$                                                      & $\ceil*{\log\frac{X_0}{\epsilon}}$         \\ \hline
Polynomial     & $a_i=\tfrac{\epsilon}{4}\cdot\left(\frac{4a_{i-1}}{\epsilon}\right)^{\frac{3}{4}}=\tfrac{\epsilon}{4}\cdot(\tfrac{4X_0}{\epsilon})^{(\frac{3}{4})^i}$ & $\ceil*{\frac{\log\log\frac{4X_0}{\epsilon}}{\log\frac{4}{3}}}$     \\ \hline
\end{tabular}
\caption{Three common ways of setting $\{a_i\}$ in an interval analysis. We specify how to pick $a_i$ according to $a_{i-1}$ in the second column and calculate the number of intervals in the third column. Note that the constants used here are arbitrarily chosen and can be further optimized during the implementation.}
\label{table:interval a}
\end{table}

\paragraph{How to set $\{\delta_i\}$.}
In general, a handy way to set $\{\delta_i\}$ is setting $\delta_i=\delta/(2i^2)$. First, note that $\sum_i\delta_i\leq\delta$ as desired. Second, in high probability bound, we usually get $\log\tfrac{1}{\delta'}$ dependency from the martingale concentration. In such case, we have $\log\tfrac{1}{\delta_i}=\log\tfrac{1}{\delta}+2\log i+1$ where $\log\tfrac{1}{\delta}$ is essential from concentration and $\log i$ is the cost of union bound. See~\autoref{app:pca} for concrete examples.

\section{Details on Streaming \texorpdfstring{$k$}{}-PCA}\label{app:pca}
Let $\cD$ be a distribution over the unit sphere in $\Real^n$ and $\Sigma=\Exp_{\bx\sim\cD}[\bx\bx^\top]$ be its covariance matrix. Given a sequence of i.i.d. samples $\bx_1,\dots,\bx_T$ from $\cD$, the goal of streaming $k$-PCA is to output a $k$ dimensional subspace that is close to the top $k$ eigenspace of $A$ using $O(nk)$ space a given $1\leq k\leq n$. 

We analyze the following Oja's algorithm~\cite{Oja82} which maintains a matrix $W_t\in\Real^{n\times k}$ at time $t$.

\vspace{0.5cm}
\begin{algorithm}[H] 
	\caption{Oja's algorithm for streaming $k$-PCA}\label{alg:PCA}
	    \textbf{Input:} Time parameter $T\in\N$, learning rate $\{\eta_t\}_{t\in\N}$, initial matrix $W_0\in\Real^{n\times k}$, and sequence of $\bx_1,\dots,\bx_T\sim\cD$.\\
	    \For{$t=1,\dots,T$}{
        Let $W_t=(1-\eta_t\bx_t\bx_t^\top)W_{t-1}$.}
	    \textbf{Output:} $QR(W_t)$, an orthonormal basis of the column space of $W_t$.
\end{algorithm}
\vspace{0.5cm}

To measure how well $W_t$ converges to the top $k$ eigenspace, it is standard to use the following objective function.
\begin{equation}\label{eq:pca potential app}
X_t=\|Z^\top QR(W_t)\|_F^2=\|Z^\top W_t(V^\top W_t)^{-1}\|_F^2
\end{equation}
where $QR(\cdot)$ stands for the QR decomposition and $V$ (resp. $Z$) is an orthogonal basis for the eigenspace corresponds to eigenvalues $\lambda_1,\dots,\lambda_k$ (resp. $\lambda_{k+1},\dots,\lambda_n$). Denote $\Sigma=\Exp_{\bx\sim\cD}[\bx\bx^\top]$, $\Sigma_{\leq k}=V\diag(\lambda_1,\dots,\lambda_k)V^\top$, and $\Sigma_{>k}=Z\diag(\lambda_{k+1},\dots,\lambda_n)Z^T$. The goal is to show that $X_t$ converges to $0$ efficiently.

There are two common convergence guarantees for $k$-PCA. The \textit{local convergence} which starts from a good initialization such that $X_0\leq1$ and the \textit{global convergence} where $W_0$ is randomly chosen. On the other hand, it is also common to consider the following two eigengap settings: the \textit{gap-dependent setting} which assumes $\gap=\lambda_k-\lambda_{k+1}>0$ and the \textit{gap-free setting} where the goal is showing that $W_t$ is close to the top eigenspace corresponds to eigenvalue $\geq\lambda_k-\rho$ for some parameter $\rho>0$. In this paper, since the goal is to demonstrate the power of the proposed framework, we focus on the simplest non-trivial setting: the local convergence for gap-dependent $k$-PCA. Specifically, we apply the framework and prove the following theorem.

\begin{theorem}[Local convergence for gap-dependent streaming $k$-PCA]\label{thm:pca details}
Let $\cD$ be a distribution over the unit sphere in $\Real^n$, $\lambda_1\geq\lambda_2\geq\dots\geq\lambda_n\geq0$ be the eigenvalues of its covariance matrix, and let $W_t$ be the output of the Oja's algorithm at time $t\in\N$. Let $\lambda=\lambda_1+\cdots+\lambda_k$ and $\gap=\lambda_k-\lambda_{k+1}$.
For every $\delta\in(0,1)$, there exists an adaptive learning rate such that we have the following strong uniform convergence.
\[
\Pr\left[\exists t\in\N,\ \|Z^\top W_t(V^\top W_t)^{-1}\|_F^2>\frac{1500\lambda(\log\frac{1}{\delta}+2\log\log (t+1))}{\gap^2t}\right]<\delta \, .
\]
Also, there exists an adaptive learning rate such that we have the following last iterate convergence. For every $t\in\N$
\[
\Pr\left[\|Z^\top W_t(V^\top W_t)^{-1}\|_F^2>\frac{2000\lambda\log\frac{1}{\delta}}{\gap^2t}\right]<\delta \, .
\]
\end{theorem}

\paragraph{Comparison.} The convergence rate of the Oja's algorithm is a well-studied problem~\cite{AL17,DOR15,Shamir16}. For the gap-dependent local convergence, the previous state-of-the-art analysis~\cite{AL17} gives $O(\frac{\lambda(\log\delta^{-1}+\log^3t)}{\gap^2t})$ convergence rate while in~\autoref{thm:pca details} we improve to $O(\frac{\lambda(\log\delta^{-1}+\log\log t)}{\gap^2t})$. Note that there is an information-theoretic lower bound $\Omega(\frac{k\lambda_k}{\gap^2t})$. See~\cite{AL17} for more comparisons with other previous works as well as other settings.

\paragraph{Section structure.} In the rest of this section, we provide the recursion analysis in~\autoref{app:pca moment}, the moment and concentration analysis in~\autoref{sec: pca concentration}, and the improvement analysis in~\autoref{app:pca improvement}.

\subsection{Recursion analysis}\label{app:pca moment}
Here we provide the details on the recursion analysis. The proof of this step looks relatively lengthy because the objective function $X_t$ has an inverse term $(V^\top W_t)^{-1}$. Conceptually, the proofs are straightforward by properly rearranging the terms and applying matrix inequalities (see~\autoref{lem:matrix inequalities}).

\begin{lemma}[Linearization for $k$-PCA]\label{lem:pca linearization}
For any $t\in \N$, we have
\[
X_t \leq (1 - 2\eta_t\gap)X_{t-1} + N_t
\]
with
\begin{align*}
N_t &= 2\eta_t\cdot\left(-\tr(Y_{t-1}^\top Y_{t-1}B_t)+\Exp[\tr(Y_{t-1}^\top Y_{t-1}B_t)]+\tr(Y_{t-1}^\top C_t)-\Exp[\tr(Y_{t-1}^\top C_t)]\right)\\
&+\frac{2\eta_t^2a_t}{1+\eta_ta_t}\cdot\left(\tr(Y_{t-1}^\top Y_{t-1}B_t)-\tr(Y_{t-1}^\top C_t)\right)+\frac{2\eta_t^2}{(1+\eta_ta_{t})^2}\cdot\left(\|Y_{t-1}B_t\|_F^2 +\|C_t\|_F^2\right)
\end{align*}
where
\begin{align*}
&Y_t=Z^\top W_t(V^\top W_t)^{-1}\, ,\ 
&a_t = \bx_t^\top W_{t-1}(V^\top W_{t-1})^{-1}V^\top \bx_t\, ,\\
&B_t = V^\top \bx_t\bx_t^\top W_{t-1}(V^\top W_{t-1})^{-1}\, ,\ 
&C_t = Z^\top \bx_t\bx_t^\top W_{t-1}(V^\top W_{t-1})^{-1} \, .
\end{align*}
In particular, we get the following recursion
\begin{equation}
\label{eq:pca recursion}
X_t\leq\prod_{t'=T_0+1}^t(1-2\eta_{t'}\gap)\cdot(X_{T_0}+M_t)
\end{equation}
where
\[M_t=\sum_{t'=T_0+1}^t\prod_{t''=T_0+1}^{t'}(1-2\eta_{t''}\gap)^{-1}N_{t'}.\]
\end{lemma}

\begin{proof}
First by Sherman-Morisson formula, we have
\[
(V^\top W_t)^{-1} = (V^\top W_{t-1})^{-1} - \frac{\eta_t}{1+\eta_t a_t}(V^\top W_{t-1})^{-1}B_{t} \, .
\]
Now we have
\begin{align}
Y_{t} &= Z^\top W_{t-1}(V^\top W_t)^{-1} + \eta_tZ^\top \bx_t\bx_t^\top W_{t-1}(V^\top W_t)^{-1}\nonumber\\
&= Y_{t-1} - \frac{\eta_t}{1+\eta_ta_{t}}Y_{t-1}B_{t} + \eta_tC_{t} - \frac{\eta_t^2}{1+\eta_ta_t}C_{t}B_{t} \, . \nonumber
\intertext{Because $C_tB_t = a_tC_t$, we have}
&= Y_{t-1} - \frac{\eta_t}{1+\eta_ta_{t}}Y_{t-1}B_{t} + \left(\eta_t - \frac{\eta_t^2a_t}{1+\eta_ta_t}\right)C_{t}\nonumber\\
&=  Y_{t-1} - \frac{\eta_t}{1+\eta_ta_{t}}Y_{t-1}B_{t} + \frac{\eta_t}{1+\eta_ta_{t}}C_{t}\label{eq: inverse 1} \, .
\end{align}
Next, expand the square of the Frobenius norm of~\autoref{eq: inverse 1} as follows.
\begin{align}
\|Y_t\|_F^2 & =\tr(Y_t^\top Y_t ) = \tr(Y_{t-1}^\top Y_{t-1}) - \frac{2\eta_t}{1+\eta_ta_{t}}\cdot\left( \tr(Y_{t-1}^\top Y_{t-1}{B_t}) + \tr(Y_{t-1}^\top{C_t} )\right) \nonumber\\&+ \frac{\eta_t^2}{(1+\eta_ta_{t})^2}\cdot\left(\|Y_{t-1}B_t\|_F^2 - 2\tr(Y_{t-1}^\top{C_t}B_t) + \|C_t\|_F^2\right) \, .\nonumber
\intertext{By AM-GM inequality for matrix (see~\autoref{lem:matrix inequalities}), we have}
&\leq \tr(Y_{t-1}Y_{t-1}^\top ) - \frac{2\eta_t}{1+\eta_ta_{t}}\cdot\left(\tr(Y_{t-1}^\top Y_{t-1}{B_t}) + \tr(Y_{t-1}^\top{C_t})\right) \nonumber\\&+ \frac{2\eta_t^2}{(1+\eta_ta_{t})^2}\cdot\left(\|Y_{t-1}B_t\|_F^2 + \|C_t\|_F^2\right) \, .\nonumber
\intertext{Plus and minus $2\eta_t\cdot\left(\tr(C_t^\top Y_{t-1})-\tr(Y_{t-1}^\top Y_{t-1}{B_t})\right)$, we have }
&\leq \tr(Y_{t-1}Y_{t-1}^\top ) - 2\eta_t\cdot\left( \tr(Y_{t-1}^\top Y_{t-1}{B_t}) - \tr(Y_{t-1}^\top{C_t})\right) \nonumber\\
&+\frac{2\eta_t^2a_t}{1+\eta_ta_t}\cdot\left(\tr(Y_{t-1}^\top Y_{t-1}B_t)-\tr(C_t^\top Y_{t-1})\right)+\frac{2\eta_t^2}{(1+\eta_ta_{t})^2}\cdot\left(\|Y_{t-1}B_t\|_F^2 +\|C_t\|_F^2\right) \, .\label{eq: linearization 1}
\end{align}
Finally, let us calculate the expectation of the second term as follows. Recall that $\Exp_{\bx\sim\cD}[\bx\bx^\top]=\Sigma$ is the covariance matrix.
\begin{align}
\Exp[\tr(Y_{t-1}^\top Y_{t-1}B_t) | \mathcal{F}_{t-1}] &= \tr(Y_{t-1}^\top Y_{t-1}V^\top \Sigma W_{t-1}(V^\top W_{t-1})^{-1})\nonumber\\
&= \tr(Y_{t-1}^\top Y_{t-1}\Sigma_{\leq k}V^\top W_{t-1}(V^\top W_{t-1})^{-1})\nonumber\\
&= \tr(Y_{t-1}^\top Y_{t-1}\Sigma_{\leq k})\geq \lambda_k\tr(Y_{t-1}^\top Y_{t-1})\label{eq: exp top eigen}
\intertext{and}
\Exp[\tr(Y_{t-1}^\top C_t) | \mathcal{F}_{t-1}] &= \tr(Y_{t-1}^\top Z^\top \Sigma W_{t-1}(V^\top W_{t-1})^{-1})\nonumber\\
&= \tr(Y_{t-1}^\top \Sigma_{> k}Z^\top W_{t-1}(V^\top W_{t-1})^{-1})\nonumber\\
&= \tr(Y_{t-1}^\top \Sigma_{> k}Y_{t-1})\leq \lambda_{k+1}\tr(Y_{t-1}^\top Y_{t-1}) \, .\label{eq: exp bottom eigen}
\end{align}
Combining~\autoref{eq: linearization 1},~\autoref{eq: exp top eigen}, and~\autoref{eq: exp bottom eigen}, as $X_t=\|Y_t\|_F^2$, we have
\[
X_t\leq (1-2\eta_t\gap)\cdot X_{t-1} + N_t
\]
as desired. By unrolling the recursion, we obtained~\autoref{eq:pca recursion}.
\end{proof}

\subsection{Moment and concentration analysis}
\label{sec: pca concentration}
In this section, we give the following lemma provides the moment profile for $\{X_t\}$.

\begin{lemma}[Moment profile for $k$-PCA]\label{lem:pca moment details}
Consider the setting in~\autoref{thm:pca details} with learning rate $0<\eta_t\leq1/10$ for all $t\in\N$. Let $0.01\geq \Lambda_t > 0$ and $\tau$ is the stopping time for the event $\lbrace X_t > \Lambda_t \rbrace$. For every $T_0+1\leq t\leq T$, we have the following.
\begin{itemize}
\item (Bounded difference) $\left|\mathbf{1}_{\tau\geq t}N_t\right|  \leq\eta_t\cdot(5\sqrt{\Lambda_t}+\eta_t)$.
\item (Conditional expectation) $\left|\Exp\left[\mathbf{1}_{\tau\geq t}N_t\, |\, \mathcal{F}_{t-1}\right]\right|\leq 2\eta_t^2\lambda$.
\item (Conditional variance) $\left|\Var\left[\mathbf{1}_{\tau\geq t}N_t\, |\, \mathcal{F}_{t-1}\right]\right|\leq 45\eta_t^2\lambda\Lambda + 30\eta_t^4\lambda$.
\end{itemize}
Specifically, if $\eta_t = \gamma/ 2\gap$, the following functions $(B_{T_0},\mu_{T_0},\sigma^2_{T_0})$ form a moment profile for $\{M_t\}$ and $\Lambda$

\[
B_{T_0}(t,\Lambda) = \frac{10\gap\gamma\sqrt{\Lambda} + \gamma^2}{4\gap^2(1-\gamma)^{t-T_0}}\, ,~\mu_{T_0}(t,\Lambda)= \frac{\lambda\gamma^2}{2\gap^2 (1-\gamma)^{t-T_0}}\,  \text{, and }
\sigma_{T_0}^2(t,\Lambda) =  \frac{12\gamma^2\lambda\Lambda\gap^2 + 2\gamma^4\lambda}{\gap^4(1-\gamma)^{2(t-T_0)}}.\, 
 \]
\end{lemma}

\begin{proof}
Let us start with three useful facts we are going to heavily apply throughout the proof. First, $I=VV^\top+ZZ^\top$ because $V$ and $Z$ form an orthonormal eigenbasis for the covariance matrix. Second, the following show that $\mathbf{1}_{\tau\geq t}a_t$ is small almost surely.
\begin{align*}
|\mathbf{1}_{\tau\geq t}a_t|&=|\mathbf{1}_{\tau\geq t}\bx_t^\top(VV^\top+ZZ^\top)W_{t-1}(V^\top W_{t-1})^{-1}V^\top \bx_t|\\
&\leq\|V^\top\bx_t\|_2^2+|\mathbf{1}_{\tau\geq t}\bx_t^\top ZY_{t-1}V^\top\bx_t| \, .
\intertext{Because $\|\bx_t\|_2=1$, we have}
&\leq1+\|\mathbf{1}_{\tau\geq t}Y_{t-1}\|_F\leq 1.1 \, .
\end{align*}
Third, let $A_t = x_t^\top W_{t-1}(V^\top W_{t-1})^{-1}$. We have
\begin{align}
\|\mathbf{1}_{\tau\geq t}A_{t-1}\|_2 &= \|\mathbf{1}_{\tau\geq t}x_t^\top(VV^\top + ZZ^\top)W_{t-1}(V^\top W_{t-1})^{-1}\|_2\nonumber\\
&\leq \|x_t^\top V\|_2 + \|\mathbf{1}_{\tau\geq t}x_t^\top ZY_{t-1}\|_2\nonumber\\
&\leq 1 + \|\mathbf{1}_{\tau\geq t}Y_{t-1}\|_F < 1.1 \, .
\end{align}
Thus, as $\eta_t\leq 0.1$, we have
\begin{align*}
\mathbf{1}_{\tau\geq t}\frac{2\eta_t^2a_t}{1+\eta_ta_t}\leq3\eta_t^2 \ \text{ and }\ \mathbf{1}_{\tau\geq t}\frac{2\eta_t^2}{(1+\eta_ta_t)^2}\leq3\eta_t^2 \, .
\end{align*}
\begin{itemize}
\item (Bounded difference) First, observe that $\|\mathbf{1}_{\tau\geq t}B_t\|_F$ and $\|\mathbf{1}_{\tau\geq t}C_t\|_F$ are small almost surely. Concretely, by Cauchy-Schwarz inequality and the fact that $B_t, C_t$ are rank $1$ matrix, we have
\begin{align*}\label{eq: BC bounded diff}
 &\|\mathbf{1}_{\tau\geq t}B_t\|_F = \|\mathbf{1}_{\tau\geq t}B_t\| = \|V^\top x_t\|_2\|x_t^\top A_{t-1}\|_2 < 1.1\, ,\\ &\|\mathbf{1}_{\tau\geq t}C_t\|_F =\|\mathbf{1}_{\tau\geq t}C_t\| = \|Z^\top x_t\|_2\|x_t^\top A_{t-1}\|_2 < 1.1\, .
\end{align*}
Thus, by matrix Cauchy-Schwarz inequality (see~\autoref{lem:matrix inequalities}), we have
\begin{align*}
|\mathbf{1}_{\tau\geq t}\tr(Y_{t-1}^\top C_t)|&\leq\|\mathbf{1}_{\tau\geq t}Y_{t-1}\|_F\cdot\|\mathbf{1}_{\tau\geq t}C_t\|_F\leq 1.1\sqrt{\Lambda_t} \, .
\end{align*}
By matrix Holder's inequality (see~\autoref{lem:matrix inequalities}), the fact that $\|AA^\top\|_1 = \|A\|_F^2, \|A\|_\infty = \|A\|$ and for rank $1$ matrix $\|AA^T\| = \|A\|^2 = \|A\|_F^2$, we have

\begin{align*}
&|\mathbf{1}_{\tau\geq t}\tr(Y_{t-1}^\top Y_{t-1}B_t)|\leq\mathbf{1}_{\tau \geq t}\|Y_{t-1}^\top Y_{t-1}\|_1\|B_t\|_\infty = \mathbf{1}_{\tau \geq t}\|Y_{t-1}\|_F^2\|B_t\|_F\leq1.1\Lambda_t  \, \text{, and}\\
&\|\mathbf{1}_{\tau\geq t}Y_{t-1}B_t\|_F^2\leq\mathbf{1}_{\tau \geq t}\|Y_{t-1}Y_{t-1}^\top\|_1\|B_tB_t^\top\|_\infty\leq \mathbf{1}_{\tau \geq t}\|Y_{t-1}\|_F^2\|B_t\|^2_F\leq1.21\Lambda_t  \, .
\end{align*}

By triangle inequality and $\eta_t\leq1/10$, we have
\[
|\mathbf{1}_{\tau\geq t}N_t|\leq\eta_t\cdot(5\sqrt{\Lambda_t}+\eta_t)   \, .
\]

\item (Conditional expectation) By reusing the inequalities in the calculation of bounded difference, we have
\begin{align*}
&\left|\mathbf{1}_{\tau\geq t}\Exp\left[\frac{2\eta_t^2a_t}{1+\eta_ta_t}\cdot\tr(Y_{t-1}^\top Y_{t-1}B_t)+\frac{2\eta_t^2}{(1+\eta_ta_t)^2}\cdot(\|Y_{t-1}B_t\|_F^2+\|C_t\|_F^2)\, |\, \cF_{t-1}\right]\right|\\
\leq\ &\eta_t^2\Lambda\cdot\left(3\Exp\left[\|\mathbf{1}_{\tau\geq t}a_tB_t\|_F\, |\, \cF_{t-1}\right]+3\Exp\left[\|\mathbf{1}_{\tau\geq t}B_t\|_F^2\, |\, \cF_{t-1}\right]+3\Exp\left[\|\mathbf{1}_{\tau\geq t}C_t\|_F^2\, |\, \cF_{t-1}\right] \right) \, .
\end{align*}
Note that $a_t^2=\tr(B_t^\top B_t)$ and thus $|a_t|=\|B_t\|_F$. Namely, $\|a_tB_t\|_F=\|B_t\|_F^2$. Also, by matrix inequalities (see~\autoref{lem:matrix inequalities}), we have
\begin{align*}
\Exp[\|\mathbf{1}_{\tau\geq t}B_t\|_F^2\, |\, \cF_{t-1}]&\leq \Exp[\mathbf{1}_{\tau\geq t}\|x_t^\top A_{t-1}\|_F^2\|V^\top x_t\|_F^2\, |\, \cF_{t-1}]\\
&\leq 1.21\Exp[\|V^\top x_t\|_F^2\, |\, \cF_{t-1}] = 1.21\lambda
\end{align*}
and
\begin{align}
\Exp[\|\mathbf{1}_{\tau\geq t}C_t\|_F^2 | \mathcal{F}_{t-1}] &\leq \Exp[\mathbf{1}_{\tau\geq t}\|x_tA_{t-1}\|_F^2\|Z^\top x_t\|_F^2 | \mathcal{F}_{t-1}]\nonumber\\
&\leq \Exp[\mathbf{1}_{\tau\geq t}\|x_tA_{t-1}\|_F^2 | \mathcal{F}_{t-1}]\nonumber\\
&= \mathbf{1}_{\tau\geq t}Tr(A_{t-1}^T\Sigma A_{t-1})\nonumber\\
&= Tr(A_{t-1}^T \Sigma_{\leq k}  A_{t-1}) + \mathbf{1}_{\tau\geq t}Tr(A_{t-1}^T\Sigma_{> k} A_{t-1})\nonumber\\
&\leq 1.21\lambda_k + 1.21\lambda_{k+1} \leq 2.5\lambda \, .\nonumber
\end{align}

As for the $\tr(Y_{t-1}^\top C_t)$ term, use the identity $I=VV^\top+ZZ^\top$, we have
\begin{align*}
|\tr(Y_{t-1}^\top C_t)|&=|\tr(Y_{t-1}^\top Z^\top\bx_t\bx_t^\top W_{t-1}(V^\top W_{t-1})^{-1})|\\
&=|\tr(Y_{t-1}^\top Z^\top\bx_t\bx_t^\top(VV^\top+ZZ^\top)W_{t-1}(V^\top W_{t-1})^{-1})|\\
&\leq|\tr(Y_{t-1}^\top Z^\top\bx_t\bx_t^\top V)|+|\tr(Y_{t-1}^\top Z^\top\bx_t\bx_t^\top ZY_{t-1})| \, .
\intertext{By the matrix AM-GM inequality, we have}
&\leq\frac{1}{2}\|V^\top\bx_t\|_2^2+\frac{3}{2}\tr(Y_{t-1}^\top Z^\top\bx_t\bx_t^\top ZY_{t-1}) \, .
\end{align*}
Thus,
\begin{align*}
\Exp\left[\mathbf{1}_{\tau\geq t}\left|\frac{2\eta_t^2a_t}{1+\eta_ta_t}\tr(Y_{t-1}^\top C_t)\right|\, |\, \cF_{t-1}\right]&\leq \mathbf{1}_{\tau\geq t} 1.5\eta_t^2\cdot\Exp\left[\|V^\top\bx_t\|_2^2+3\tr(Y_{t-1}^\top Z^\top\bx_t\bx_t^\top ZY_{t-1}) \, |\, \cF_{t-1}\right]\\
&\leq \mathbf{1}_{\tau\geq t}1.5\eta_t^2\cdot(\lambda+3\lambda_{k+1}X_{t-1})\\
&\leq 1.5\eta_t^2\cdot(\lambda+3\lambda_{k+1}\Lambda)\, \leq 1.6\eta_t^2\lambda .
\end{align*}
To sum up, we have
\[
\Exp[|\mathbf{1}_{\tau\geq t}N_t| \, |\, \cF_{t-1}]\leq2\eta_t^2\lambda \, .
\]

\item (Conditional variance) Let us start with a rough estimation as follows.
\begin{align*}
\Var[ |\mathbf{1}_{\tau\geq t}N_t|\, |\, \cF_{t-1}]&\leq17\eta_t^2\mathbf{1}_{\tau\geq t}\cdot\Exp\left[\tr(Y_{t-1}^\top Y_{t-1}B_t)^2+\tr(Y_{t-1}^\top C_t)^2\, |\, \cF_{t-1}\right]\\
&+11\eta_t^4\mathbf{1}_{\tau\geq t}\cdot\Exp\left[\|Y_{t-1}B_t\|_F^4+ \|C_t\|_F^4 \, |\, \cF_{t-1}\right] \, .
\end{align*}
By reusing the previous calculation, we have
\begin{align*}
&\Exp[\mathbf{1}_{\tau\geq t}\tr(Y_{t-1}^\top Y_{t-1}B_t)^2\, |\, \cF_{t-1}]\leq \Exp[\mathbf{1}_{\tau\geq t}\|B_t\|_F^2\|Y_{t-1}\|_F^4\, |\, \cF_{t-1}]\leq 2\lambda \Lambda_t^2\leq 0.02\lambda\Lambda_t \, ,\\
&\Exp[\mathbf{1}_{\tau\geq t}\tr(Y_{t-1}^\top C_t)^2\, |\, \cF_{t-1}]\leq \Exp[\mathbf{1}_{\tau\geq t}\|Y_{t-1}\|_F^2\|C_t\|_F^2\, |\, \cF_{t-1}]\leq 2.5\lambda \Lambda\, ,\\
&\Exp[\mathbf{1}_{\tau\geq t}\|Y_{t-1}B_t\|_F^2\, |\, \cF_{t-1}]\leq\Exp[\mathbf{1}_{\tau\geq t}\|Y_{t-1}\|_F^2\|B_t\|^2_F\, |\, \cF_{t-1}]\leq 2\lambda \Lambda \, \text{, and}\\
&\Exp[\mathbf{1}_{\tau\geq t}\|C_t\|_F^2\, |\, \cF_{t-1}]\leq2.5\lambda \, .
\end{align*}

Thus, we have
\[
\Var[\|\mathbf{1}_{\tau\geq t}N_t|\, |\, \cF_{t-1}]\leq 45\eta_t^2\lambda\Lambda+ 30\eta_t^4\lambda \, .
\]
\end{itemize}
\end{proof}

\subsection{Improvement analysis}\label{app:pca improvement}

\subsubsection{Strong uniform convergence}

\begin{lemma}[Improvement analysis for $k$-PCA]\label{lem:pca concentration details}
Let $\{X_t\}$ be the stochastic process described in~\autoref{eq:pca potential app}. For every $A_0>A_1>0$, $\delta'>0$, $0<\Lambda\leq1$, and $1\leq T_0<T_1\in\N$, let
\[
\Delta=\frac{\gamma\lambda\log\frac{1}{\delta'}}{\gap^2 (1-\gamma)^{T_1 - T_0}}\left(\sqrt{\frac{48\gap^2\Lambda}{\gamma\lambda\log\frac{1}{\delta'}}} + \sqrt{\frac{8\gamma}{\lambda\log\frac{1}{\delta'}}} + 12\right) \, .
\]
Suppose that we have $A_0 + \Delta < \Lambda$ and $(1-\gamma)^{T_1 - T_0} \Lambda < A_1$. Then, 
\[
\Pr\left[\exists{T_0 + 1\leq t\leq T_1} \, ,\ X_t > (1-\gamma)^{t - T_0} \Lambda \ \middle| \ X_{T_0} \leq A_0\right] < \delta' \, .
\]
In particular, the above implies $\Pr[X_{T_1}>A_1\ |\ X_{T_0}\leq A_0]\leq\delta'$.
\end{lemma}
\begin{proof}

\[
B_{T_0}(t,\Lambda) = \frac{10\gap\gamma\sqrt{\Lambda} + \gamma^2}{4\gap^2(1-\gamma)^{t-T_0}}\, ,~\mu_{T_0}(t,\Lambda)= \frac{\lambda\gamma^2}{2\gap^2 (1-\gamma)^{t-T_0}}\,  \text{, and }
\sigma_{T_0}^2(t,\Lambda) =  \frac{12\gamma^2\lambda\Lambda\gap^2 + 2\gamma^4\lambda}{\gap^4(1-\gamma)^{2(t-T_0)}}.\, 
 \]

Given the moment profile $(B_{T_0},\mu_{T_0},\sigma^2_{T_0})$ in~\autoref{lem:pca moment details}, we apply~\autoref{lem:martingale concentration} and get the following bound for the deviation. 
\begin{align*}
 &2\max\left\lbrace\sqrt{\sum_{t = T_0 + 1}^{T_1}\sigma_{T_0}^2(t, \Lambda)\log\frac{1}{\delta'}}, 2\max_{T_0+1\leq t\leq T_1}B_{T_0}(t,\Lambda)\log\frac{1}{\delta'}\right\rbrace + \sum_{t'=T_0+1}^{T_1}\mu_{T_0}(t, \Lambda)\\
 =&\ \max\left\{2\sqrt{\sum_{t = T_0 + 1}^{T_1}\frac{12\gamma^2\lambda\Lambda\gap^2 + 2\gamma^4\lambda}{\gap^4(1-\gamma)^{2(t-T_0)}}\log\frac{1}{\delta'}} , \frac{10\gap\gamma\sqrt{\Lambda} + \gamma^2}{\gap^2(1-\gamma)^{t-T_0}}\log\frac{1}{\delta'}\right\} + \sum_{t=T_0+1}^{T_1}\frac{\lambda\gamma^2}{2\gap^2 (1-\gamma)^{t-T_0}} \, .\\
  \leq&\ \max\left\{2\sqrt{\frac{12\gamma\lambda\Lambda\gap^2 + 2\gamma^3\lambda}{\gap^4(1-\gamma)^{2(T_1-T_0)}}\log\frac{1}{\delta'}} , \frac{10\gap\gamma\sqrt{\Lambda} + \gamma^2}{\gap^2(1-\gamma)^{T_1-T_0}}\log\frac{1}{\delta'}\right\} + \frac{\lambda\gamma}{\gap^2 (1-\gamma)^{T_1-T_0}} \, .\\
 \leq& \frac{\gamma\lambda\log\frac{1}{\delta'}}{\gap^2 (1-\gamma)^{T_1 - T_0}}\left(\sqrt{\frac{48\gap^2\Lambda}{\gamma\lambda\log\frac{1}{\delta'}}} + \sqrt{\frac{8\gamma}{\lambda\log\frac{1}{\delta'}}} + 12\right)
\end{align*}

Next, by~\autoref{lem:martingale concentration} we get the desiring improvement inequalities.
\end{proof}

Finally, we perform an interval analysis and complete the proof of~\autoref{thm:pca details}.
\begin{proof}
Let $\ell=\infty$ and let $t_0 = 0, a_0 = 1$. For each $i\geq1$, let
\[
\delta_i=\frac{\delta}{2i^2},\ a_i = 2^{-i},\ \gamma_i = \frac{a_{i-1}\gap^2}{1500\lambda \log\frac{1}{\delta_i}},\ t_i = t_{i-1} + \ceil*{\frac{-2}{\log(1 - \gamma_i)}},\ \Lambda_i = 2a_{i-1} \, ,
\]
and $\eta_t=\gamma_i/2\gap$ for every $t_{i-1}+1\leq t\leq t_i$.
Let $T = t_{\ell}$. Observed that due to the choice of the parameters we have $a_\ell \leq \epsilon$ and $1/5\leq(1-\gamma_i)^{t_i-t_{i-1}}\leq1/4$. Now, for each $i=1,2,\dots,\ell$, we invoke~\autoref{lem:pca concentration details} with $A_0=a_{i-1}$, $A_1=a_i$, $T_0=t_{i-1}$, $T_1=t_i$, and $\delta'=\delta_i$.
Let us verify the two conditions. First, we verify the pull out condition as follows.
\begin{align*}
a_{i-1} + \Delta_i
&=a_{i-1} + \frac{\gamma_i\lambda\log\frac{1}{\delta_i}}{\gap^2(1-\gamma_i)^{t_i-t_{i-1}}}\left(\sqrt{\frac{48\gap^2\Lambda}{\gamma\lambda\log\frac{1}{\delta'}}} + \sqrt{\frac{8\gamma}{\lambda\log\frac{1}{\delta'}}} + 12\right)\\
&\leq a_{i-1} + \frac{5a_{i-1}}{1500}\cdot\left(\sqrt{48\times 1500}+13\right)\\
& < 2a_{i-1} < \Lambda_i \, .
\end{align*}
For the improvement condition, we have
\[
 (1-\gamma)^{t_i - t_{i-1}} \cdot 2a_{i-1} < \frac{2a_{i-1}}{4} = a_i \, .
\]
By~\autoref{lem:pca concentration details} and union bounding over the intervals, we have
\[
\Pr[\exists i\in\N,\ t_{i-1}+1\leq t\leq t_i \text{ s.t. } X_{t}> \Lambda_i] < \sum_{i=1}^\infty\delta_i\leq\delta \, .
\]
To have an explicit upper bound for the convergence rate, note that 
\[
\ceil*{\frac{-2}{\log(1 - \gamma_i)}}\leq 1 +\frac{2}{\gamma_i}=\frac{3000\lambda(\log\frac{1}{\delta}+\log 2i^2)\cdot 2^{i}}{\gap^2} \, .
\]
So $t_i \leq \frac{3000\lambda(\log\frac{1}{\delta}+\log 2i^2)2^{i+1}}{\gap^2}$. Now, for every $t\in\N$, let $i=\ceil*{\log\tfrac{t\gap^2}{6000\lambda(\log\tfrac{1}{\delta}+2\log\log(t+2))}}$. We have $t\geq t_{i-1}$. We also have $\Lambda_i=2a_{i-1}=2^{-i-2}\leq\frac{1500\lambda(\log\frac{1}{\delta}+2\log\log(t+2))}{\gap^2 t}$. Finally, the following strong uniform convergence holds.
      \[
\Pr\left[\exists t\in\N,\ X_t>\frac{1500\lambda(\log\tfrac{1}{\delta}+2\log\log (t+1))}{\gap^2t}\right] < \delta \, .
\]
This completes the first part of~\autoref{thm:pca details}.

\end{proof}

\subsubsection{Last iterate convergence}
Let $t_0=0$, $\gamma_1=\frac{\gap^2}{100000\lambda \log\frac{1}{\delta}}$, $t_1=\ceil*{\frac{-2}{\log(1 - \gamma_1)}}$, and $a_1 = 1$. For each $i\geq2$, let
\[
a_i = 2^{-i}a_0,\ \gamma_i = 2^{-i}\gamma_0,\ t_i = 2^it_0 \, ,
\]
and $\eta_t=\gamma_i/2\gap$ for every $t_{i-1}+1\leq t\leq t_i$.

Let us first see some basic properties from the above choice of parameters:
\begin{itemize}
\item For every $i\in\N$, we have
\[
\prod_{t'=t_0+1}^{t_i}(1-2\eta_{t'}\gap)\leq\frac{1}{4^i}=\frac{t_0^2}{t_i^2} \, .
\]
\item For every $i\in\N$, we have
\[
\sum_{t=t_{i-1}+1}^{t_i}\frac{1}{\prod_{t'=t_0+1}^t(1-2\eta_{t'}\gap)}\leq\frac{4^i}{\gamma_i} \, .
\]
\item For every $i\in\N$, we have
\[
\sum_{t=t_{i-1}+1}^{t_i}\frac{1}{\prod_{t'=t_0+1}^t(1-2\eta_{t'}\gap)^2}\leq\frac{4^{2i}}{2\gamma_i} \, .
\]
\end{itemize}

In the analysis for last iterate convergence, we consider a different stopping time $\tau$ defined for the event $\{X_t>\Lambda_t\}$ for some sequence of non-negative thresholds $\{\Lambda_t\}$ chosen properly later. Let $(B_{T_0},\mu_{T_0},\sigma^2_{T_0})$ be the moment profile obtained from~\autoref{lem:pca moment details}. By concentration inequality (i.e.,~\autoref{lem:martingale concentration}), we have
\[
\Pr\left[\exists T_0+1\leq t\leq T_1,\ |M_{t\wedge\tau}|>\Delta\right]<\delta
\]
where
\[
\Delta=2\max\left\{\sqrt{\sum_{t=T_0+1}^{T_1}\sigma^2_{T_0}(t,\Lambda_t)\log\frac{1}{\delta}},\ \max_{T_0+1\leq t\leq T_1}B_{T_0}(t,\Lambda_t)\log\frac{1}{\delta}\right\}+\sum_{t=T_0+1}^{T_1}\mu_{T_0}(t,\Lambda_t) \, .
\]
Let us focus on the case where $T_0=t_0$ and $T_1=t_\ell$ for some $\ell$. The deviation $\Delta$ can be written out as follows.
\begin{align*}
\Delta&\leq2\max\left\{\sqrt{\sum_{t=T_0+1}^{T_1}\frac{\eta_t^2\lambda\cdot(512\Lambda_t^2+624\Lambda_t+512\eta_t^2)}{\prod_{t'=T_0+1}^t(1-2\eta_{t'}\gap)^2}\log\frac{1}{\delta}},\ \max_{T_0+1\leq t\leq T_1}\frac{\eta_t\cdot(20\Lambda_t+12\sqrt{\Lambda_t}+32\eta_t)}{\prod_{t'=T_0+1}^t(1-2\eta_{t'}\gap)}\log\frac{1}{\delta}\right\}\\
&+\sum_{t=T_0+1}^{T_1}\frac{\eta_t^2\lambda\cdot(52\Lambda_t+4)}{\prod_{t'=T_0+1}^t(1-2\eta_{t'}\gap)}\\
&\leq2\max\Bigg\{\sqrt{\sum_{i=1}^\ell\frac{4^{2i}}{2\gamma_i}\cdot\frac{\gamma_i^2\lambda\cdot(512\Lambda_i^2+624\Lambda_i+128\gamma_i^2/\gap^2)}{4\gap^2}\log\frac{1}{\delta}},\\ &\max_{1\leq i\leq\ell}\frac{4^i\gamma_i\cdot(20\Lambda_i+12\sqrt{\Lambda_i}+16\gamma_i/\gap)}{2\gap}\log\frac{1}{\delta}\Bigg\}+\sum_{i=1}^\ell\frac{4^i}{\gamma_i}\cdot\frac{\gamma_i^2\lambda\cdot(52\Lambda_i+4)}{4\gap^2} \, .
\intertext{Let $\Lambda_i=\frac{T_1}{t_i^2}\Lambda=\frac{T_1}{4^it_0^2}\Lambda$ for some $\Lambda$ chosen later, we further have}
&\leq2\sqrt{\sum_{i=1}^\ell\frac{4^{2i}}{2}\cdot\frac{2^{-i}\gamma_0\lambda\cdot\left(512\frac{T_1^2\Lambda^2}{4^{2i}t_0^4}+624\frac{T_1\Lambda}{4^it_0^2}+128\cdot4^{-i}\gamma_0^2/\gap^2\right)}{4\gap^2}\log\frac{1}{\delta}}\\
&+2\max_{1\leq i\leq\ell}\frac{4^i\cdot2^{-i}\gamma_0\cdot\left(20\frac{T_1\Lambda}{4^it_0^2}+12\sqrt{\frac{T_1\Lambda}{4^it_0^2}}+16\cdot\frac{2^{-i}\gamma_0}{\gap}\right)}{2\gap}\log\frac{1}{\delta}\\
&+\sum_{i=1}^\ell\frac{4^i\cdot2^{-i}\gamma_0\lambda\cdot(52\frac{T_1\Lambda}{4^it_0^2}+4)}{4\gap^2} \\
&\leq2\sqrt{\sum_{i=1}^\ell2^{-i}\frac{64\gamma_0\lambda T_1^2\Lambda^2\log\frac{1}{\delta}}{\gap^2t_0^4}}+2\sqrt{\sum_{i=1}^\ell2^i\frac{80\gamma_0\lambda T_1\Lambda\log\frac{1}{\delta}}{\gap^2t_0}}+2\sqrt{\sum_{i=1}^\ell2^i\frac{16\gamma_0^3\lambda\log\frac{1}{\delta}}{\gap^4}}\\
&+\frac{\gamma_0\cdot\left(\frac{20T_1\Lambda}{t_0^2}+12\sqrt{\frac{T_1\Lambda}{t_0^2}}+\frac{16\gamma_0}{\gap}\right)\log\frac{1}{\delta}}{\gap}\\
&+\sum_{i=1}^\ell2^{-i}\frac{13\gamma_0\lambda T_1\Lambda}{\gap^2t_0^2}+2^i\frac{\gamma_0\lambda}{\gap^2}\\
&\leq\frac{T_1\Lambda}{100t_0^2}+\sqrt{2^{i+1}\frac{T_1\Lambda}{1000t_0}}+\sqrt{\frac{2^{i+1}}{1000}}+\frac{T_1\Lambda}{100t_0^2}+\sqrt{\frac{T_1\Lambda}{100t_0^2}}+\frac{\gamma_0}{100\gap}+\frac{T_1\Lambda}{100t_0^2}+\frac{2^{i+1}}{1000} \, .
\intertext{We pick $\Lambda=\frac{1000\lambda\log\frac{1}{\delta}}{\gap^2}$ so that the above can be upper bounded by}
&\leq\frac{T_1\Lambda}{10t_0^2} \, .
\end{align*}

Now, in order to \textit{pull out} the stopping time $\tau$, we have to check the pull-out condition in~\autoref{lem:pull out} as follows.
\[
\Pr[\tau\geq t+1\, |\, M^*_{t}\leq\Delta]=1
\]
for every $T_0+1\leq t\leq T_1$. Observe that when $M_t^*\leq\Delta$, for every $T_0+1\leq t'\leq t$, there is $i$ such that $t_{i-1}+1\leq t'\leq t_i$ by the recursion we have 
\begin{align*}
X_{t'}&=\prod_{t''=T_0+1}^{t'}(1-2\eta_{t''}\gap)(X_{T_0}+M_{t'})\\
&\leq\prod_{t''=T_0+1}^{t_{i-1}}(1-2\eta_{t''}\gap)\cdot2\Delta\\
&\leq\frac{t_0^2}{t_{i-1}^2}\cdot2\Delta\\
&\leq\frac{T_1\Lambda}{5t_{i-1}^2}=\frac{4\Lambda_{t'}}{5}
\end{align*}
and hence $\tau>t'$ by the definition of $\tau$. Since the above holds for all $T_0+1\leq t'\leq t$, we have $\tau\geq t+1$ as desired. Finally, by the pull-out lemma (i.e.,~\autoref{lem:pull out}) we have
\[
\Pr\left[\exists T_0+1\leq t\leq T_1,\ |M_{t}|>\frac{100\lambda T_1\log\frac{1}{\delta}}{\gap^2T_0^2}\right]<\delta \, .
\]
Combine with the recursion, we have
\[
\Pr\left[X_{T_1}>\frac{2000\lambda\log\frac{1}{\delta}}{\gap^2T_1}\right]<\delta \, .
\]

\section{Details on Solving Linear Bandit with SGD Updates}\label{app:ql}
In this subsection, we study linear bandit with SGD dynamic. In stochastic linear bandit, there is a true parameter $\theta_* \in B(0, L_{*}) \subseteq \mathbb{R}^d$ and at each time step $t$ the agent is presented with a decision set $D_t \subseteq B(0, L) \subseteq \mathbb{R}^d$. The agent chooses an action $x_t\in D_t$ and subsequently, the agent observe the reward
\[
y_t = \theta_*^\top x_t + \epsilon_t \, .
\]
where $|\epsilon_t|\leq 1$ and $\Exp[\epsilon_t | x_{1:t}, \epsilon_{1:t-1}] = 0$. We make the bounded assumption of noise term only to simplify the presentation and the sub-Gaussian case can be handled by our framework similarly.

We emphasize that we are considering the incremental update approach, \textit{i.e.} update the estimation of the unknown parameter via SGD instead of solving the linear regression directly (the batch processing approach, \cite{abbasi2011improved,dani2008stochastic} ) . The idea of using an SGD update appeared in \cite{korda2015fast}, but their design of upper confidence bound (UCB) is heuristic, and no regret bound is provided. \cite{jun2017scalable} develops an online-to-confidence-set algorithm to achieve $O(d(T\log ^2 (T) \log (T/\delta))^{1/2})$ regret up to iterated log-factors. They use an online Newton step predictor as a sub-routine to get rid of the dependence on historical data. In contrast, we do not need any sub-routine and update the parameter directly. As a result, we both simplify the procedure and improve the regret bound. The full protocol and algorithm is described below in~\autoref{alg:LinUCB-SGD}.

\vspace{0.5cm}
\begin{algorithm}[H]
  \caption{LinUCB-SGD}
  \label{alg:LinUCB-SGD}
      \textbf{Parameters}: $\lambda > 0$, $\eta=\lambda/L^2$, $\beta_t = 288\max\left\lbrace L_*^2\lambda, \frac{d\lambda}{L^2}\log\left(1 + \frac{T}{d}\right)\log\frac{1}{\delta}\right\rbrace$.\\
      \textbf{Initialize}: $\theta_0 \gets 0$, $V_0 \gets \lambda I$ \\
      \For{round $t=1,\dots,T$}{
      $B_{t} \gets \{\theta :\lVert \theta -\theta _{t-1} \rVert _{V_{t-1}}\le \sqrt{\beta _t}\}$.\\
      Choose $x_t= 
      \underset{x\in D_t}{\text{argmax}}\underset{\theta \in B_t}{\text{max}}\left\langle x,\theta \right\rangle 
$.\\
      Observe the reward $y_t=\left< x_t,\theta _* \right> +\epsilon _t$.\\
      $\theta_{t} \gets \theta_{t-1} + A _t\left( y_t-\theta _{t-1}^{T}x_t \right) x_t $ where $A_t = \eta V_{t-1}^{-1}$.\\
      $V_{t} \gets V_{t-1} + \eta x_tx_t^T$.}
\end{algorithm}
\vspace{0.5cm}

By expanding the SGD update, we have the following dynamics.
\begin{equation}
\label{eq:bandit_update}
\theta_{t} - \theta_* = (I - A_tx_tx_t^\top )(\theta_{t-1} - \theta_*) + \epsilon_tA_tx_t \, .
\end{equation}
The goal is to minimized the regret at time $T$, defined by $R_T = \sum_{t=1}^\top (x_t^* - x_t)^\top \theta_*$, where $x_t^*$ is the optimal action at time $t$. The regret of \autoref{alg:LinUCB-SGD} is bounded by the following.

\begin{theorem}
\label{thm: bandit main regret appendix}
Setting parameters as in \autoref{alg:LinUCB-SGD}, with probability $1-\delta$, for any $\lambda >0$, $L>0$, $L_{*}>0$,
\[
R_T\leq  34\sqrt{2dT\max\left\lbrace L_{*}^2L^2, d\log\left(1 + \frac{T}{d}\right)\log\frac{1}{\delta}\right\rbrace\log\left( 1+\frac{T}{d} \right)} \, .
\]
In particular, we have $R_T=O(d\sqrt{T\log^2 T\log(1/\delta)})$.
\end{theorem}

In order to obtain the above regret bound, we follow the standard approach in~\cite{abbasi2011improved} and study the dynamic of $X_t = \|\theta_t - \theta_*\|_{V_t}^2$ using our framework. Specifically, we will obtain the following theorem.
\begin{theorem}
\label{thm: bandit main appendix}
Setting parameters as in \autoref{alg:LinUCB-SGD}, for any $\lambda > 0, L>0, L_{*}>0, T > 0$, we have
\[
\Pr[\exists t\in [T],\ X_T > 288\max\left\lbrace L_{*}^2\lambda, \frac{d\lambda}{L^2}\log\left(1 + \frac{T}{d}\right)\log\frac{1}{\delta}\right\rbrace] < \delta \, .
\]
\end{theorem}

\paragraph{Section structure.} In the rest of this section, we provide the recursion analysis in~\autoref{app:ql moment}, the moment and concentration analysis in~\autoref{app:ql moment concentration}, and the improvement analysis in~\autoref{app:ql interval}.

\subsection{Recursion analysis}\label{app:ql moment}
We would like to apply our framework on the quantity $X_t = \|\theta_t - \theta_*\|_{V_t}^2$. We have the following lemma on linearization.
\begin{lemma}[Recursion analysis for stochastic linear bandit with SGD updates]\label{lem: bandit recursion}
Consider the setting in~\autoref{thm: bandit main appendix}. Let $\eta \leq \frac{\lambda}{L^2}$. For all $t\in \N$, we have
\[
X_t \leq X_{t-1} + N_t
\]
where
\[
N_t = 2\eta \epsilon_t(\theta_{t-1} - \theta_*)^\top x_t-2\eta^3\epsilon_t(\theta_{t-1} - \theta_*)^\top x_t\|x_t\|_{V_{t-1}^{-1}}^4 + 2\epsilon_t^2\eta^2\|x_t\|^2_{V_{t-1}^{-1}} \, .
\]

\end{lemma}
\begin{proof}
First notice that since $\eta \leq \frac{\lambda}{L^2}$, $\eta \|x_t\|_{V_{t-1}^{-1}}^2\leq 1$.
By \autoref{eq:bandit_update}, we have
\[
\|\theta_t - \theta_*\|_{V_t}^2 =  \|(I - A_tx_tx_t^\top )(\theta_{t-1} - \theta_*)\|_{V_t}^2 + 2\langle (I - A_tx_tx_t^\top )(\theta_{t-1} - \theta_*), \epsilon_tA_tx_t  \rangle_{V_t} + \epsilon_t^2\|A_tx_t\|_{V_t}^2.
\]

So in total we have
\[
X_t \leq X_{t-1} + N_t
\]
where 
\[
N_t = 2\eta\epsilon_t(\theta_{t-1} - \theta_*)^\top x_t-2\eta^3\epsilon_t(\theta_{t-1} - \theta_*)^\top x_t\|x_t\|_{V_{t-1}^{-1}}^4 + 2\epsilon_t^2\eta^2\|x_t\|^2_{V_{t-1}^{-1}} \, .
\]
In particular, 
\[
\|\theta_{t} - \theta_*\|^2_{V_{ t}} \leq \|\theta_{0} - \theta_*\|_{V_{0}}^2 + \sum_{i=1}^{ t}N_i \, .
\]
\end{proof}

\subsection{Moment and concentration analysis}\label{app:ql moment concentration}

We begin with bounding the moment profile of $\{X_i\}_{i=1}^T$.

\begin{lemma}
\label{lem: bandit moments detail}
For $\Lambda>0$, let $\tau$ is the stopping time for the event $\lbrace X_t > \Lambda \rbrace$. For every $T_0+1\leq t\leq T$, the following following functions $(B_{T_0},\mu_{T_0},\sigma^2_{T_0})$ form a moment profile for $\{X_t\}$, $\Lambda$, and $T_0$.
\begin{itemize}
\item (Bounded difference) $|\mathbf{1}_{\tau\geq t}N_t| \leq B_{T_0}(t,\Lambda)= 3\eta \|x_t\|_{V_{t-1}^{-1}}\sqrt{\Lambda} + 2\eta ^2\|x_t\|_{V_{t-1}^{-1}}^2.$
\item (Conditional expectation) $\left|\Exp\left[\mathbf{1}_{\tau\geq t}N_t\, |\, \mathcal{F}_{t-1}\right]\right|\leq \mu_{T_0}(t,\Lambda)= 2\eta^2\|x_t\|_{V_{t-1}^{-1}}^2$.
\item (Conditional variance) $\left|\Var\left[\mathbf{1}_{\tau\geq t}N_t\, |\, \mathcal{F}_{t-1}\right]\right|\leq \sigma^2_{T_0}(t,\Lambda)= 18\eta^2\|x_t\|_{V_{t-1}^{-1}}^2\Lambda + 8\eta^4\|x_t\|_{V_{t-1}^{-1}}^4.$
\end{itemize}
\end{lemma}

\begin{proof} By \autoref{lem: bandit recursion} we have 
\[
N_t = 2\eta\epsilon_t(\theta_{t-1} - \theta_*)^\top x_t-2\eta^3\epsilon_t(\theta_{t-1} - \theta_*)^\top x_t\|x_t\|_{V_{t-1}^{-1}}^4 + 2\epsilon_t^2\eta^2\|x_t\|^2_{V_{t-1}^{-1}} \, .
\]

For the first term, we have
\begin{align*}
&\|(I - A_tx_tx_t^\top )(\theta_{t-1} - \theta_*)\|_{V_t}^2\\
=&\ \|(\theta_{t-1} - \theta_*)\|_{V_t}^2 - 2\langle \theta_{t-1} - \theta_*, A_tx_tx_t^\top (\theta_{t-1} - \theta_*) \rangle_{V_t} +  \|A_tx_tx_t^\top (\theta_{t-1} - \theta_*)\|_{V_t}^2\\
=&\ \|(\theta_{t-1} - \theta_*)\|_{V_{t-1}}^2 + \eta ((\theta_{t-1}-\theta_*)^\top x_t)^2 - 2\eta ((\theta_{t-1}-\theta_*)^\top x_t)^2 -2\eta^2((\theta_{t-1}-\theta_*)^\top x_t)^2\|x_t\|_{V_{t-1}^{-1}}^2\\
+&\ \eta^2((\theta_{t-1}-\theta_*)^\top x_t)^2\|x_t\|_{V_{t-1}^{-1}}^2 + \eta^3((\theta_{t-1}-\theta_*)^\top x_t)^2\|x_t\|_{V_{t-1}^{-1}}^4\\
\leq&\ \|(\theta_{t-1} - \theta_*)\|_{V_{t-1}}^2 +  (\eta^3\|x_t\|_{V_{t-1}^{-1}}^4 - \eta^2\|x_t\|_{V_{t-1}^{-1}}^2 - \eta)((\theta_{t-1}-\theta_*)^\top x_t)^2 \\
\leq&\ \|\theta_{t-1} - \theta_*\|_{V_{t-1}}^2 \, .
\end{align*}
For the second term we have
\begin{align*}
    &2\langle (I - A_tx_tx_t^\top )(\theta_{t-1} - \theta_*), \epsilon_tA_tx_t  \rangle_{V_t}\\
    =&\ 2\langle \theta_{t-1} - \theta_*, \epsilon_tA_tx_t  \rangle_{V_t} - 2\langle  A_tx_tx_t^\top (\theta_{t-1} - \theta_*), \epsilon_tA_tx_t  \rangle_{V_t}\\
    =&\  2\eta\epsilon_t(\theta_{t-1} - \theta_*)^\top x_t + 2\eta^2\epsilon_t(\theta_{t-1} - \theta_*)^\top x_t\cdot\|x_t\|_{V_{t-1}^{-1}}^2 - 2\eta^2\epsilon_t(\theta_{t-1} - \theta_*)^\top x_t\cdot\|x_t\|_{V_{t-1}^{-1}}^2\\
    -&\ 2\eta^3\epsilon_t(\theta_{t-1} - \theta_*)^\top x_t\cdot\|x_t\|_{V_{t-1}^{-1}}^4\\
    =&\  2\eta\epsilon_t(\theta_{t-1} - \theta_*)^\top x_t-2\eta^3\epsilon_t(\theta_{t-1} - \theta_*)^\top x_t\cdot\|x_t\|_{V_{t-1}^{-1}}^4 \, .
\end{align*}
For the third term, we have
\[
\epsilon_t^2\|A_tx_t\|^2_{V_t} = \epsilon_t^2\eta^2\|x_t\|^2_{V_{t-1}^{-1}} + \epsilon_t^2\eta^3\|x_t\|^4_{V_{t-1}^{-1}} \leq 2\epsilon_t^2\eta^2\|x_t\|^2_{V_{t-1}^{-1}} \, .
\]

Putting them together, we can compute the moment profile directly. For the bounded difference by $\eta \|x_t\|_{V_{t-1}^{-1}}^{2} \le \frac{1}{2}$ and Cauchy-Schwarz inequality we have
\[
|\mathbf{1}_{\tau\geq t}N_t| \leq 3\eta\|x_t\|_{V_{t-1}^{-1}}\sqrt{\Lambda} + 2\eta^2\|x_t\|_{V_{t-1}^{-1}}^2 \, .
\]
For the conditional expectation, we have
\[
\left|\Exp\left[\mathbf{1}_{\tau\geq t}N_t\, |\, \mathcal{F}_{t-1}\right]\right| \leq 2\eta^2\|x_t\|_{V_{t-1}^{-1}}^2 \, .
\]
For the conditional variance, we have
\[
\left|\Var\left[\mathbf{1}_{\tau\geq t}N_t\, |\, \mathcal{F}_{t-1}\right]\right| \leq 18\eta^2\|x_t\|_{V_{t-1}^{-1}}^2\Lambda + 8\eta^4\|x_t\|_{V_{t-1}^{-1}}^4 \, .
\]
\end{proof}

Notice that we need to bound $\|x_t\|_{V_{t-1}^{-1}}^2$. So we begin with following helper lemma. 

\begin{lemma}[Log-determinant Lemma, see Lemma 11 in \cite{abbasi2011improved}]
\label{lem:sum_x}
We have
\[
\log \left(\frac{\det(V_{ t})}{\det(V_{0})}\right) \leq \sum_{i=1}^{t}\eta\|x_i\|_{V_{i-1}^{-1}}^2 \leq 2\log \left(\frac{\det(V_{t})}{\det(V_{0})}\right) \le 2d\log \left( 1+\frac{\eta tL^2}{d\lambda} \right) \, .
\]
\end{lemma}

\subsection{Improvement Analysis}\label{app:ql interval}

Now we can bound the deviation induced by the noise. 
\begin{lemma}[Improvement analysis for linear bandit with SGD updates]\label{lem:bandit concentration details}
Consider the setting in~\autoref{thm: bandit main appendix}. For every $A_0>A_1>0$, $\delta'>0$, $\Lambda>0$, and $1\leq T_0<T_1\in\N$, let
\[
\Delta=\sqrt{144\eta d\log\left(1 + \frac{\eta T_1L^2}{d\lambda}\right)\Lambda\log\frac{1}{\delta'}} + 16\eta d\log\left(1 + \frac{\eta T_1L^2}{d\lambda}\right)\sqrt{\log\frac{1}{\delta'}}  \, .
\]
Suppose that we have $A_0 + \Delta < \Lambda$ and $\Lambda < A_1$. Then, 
\[
\Pr\left[\max_{T_0 + 1\leq t\leq T_1}X_t >  \Lambda \ \middle| \ X_{T_0} \leq A_0\right] < \delta' \, .
\]
In particular, the above implies $\Pr[X_{T_1}>A_0\ |\ X_{T_0}\leq a]\leq\delta'$.
\end{lemma}
\begin{proof}
Given the moment profile $(B_{T_0},\mu_{T_0},\sigma^2_{T_0})$ in \autoref{lem: bandit moments detail}, we apply a martingale concentration inequality (see Theorem 6.2 in \cite{Fan2006}) to obtain the deviation of $\sum_{t=T_0}^{T_1\wedge\tau}N_t$ as follows.
\begin{align*}
 &\sqrt{\sum_{t = T_0 + 1}^{T_1}2\sigma_{T_0}^2(t, \Lambda)\log\frac{1}{\delta'} + 2B_{T_0}^2(t,\Lambda)\log\frac{1}{\delta'}} + \sum_{t=T_0+1}^{T_1}\mu_{T_0}(t, \Lambda)\\
 \leq&\ \sqrt{\sum_{t = T_0 + 1}^{T_1}\left(72\eta^2\|x_t\|_{V_{t-1}^{-1}}^2\Lambda + 32\eta^4\|x\|_{V_{t-1}^{-1}}^4\right)\log\frac{1}{\delta'} } + \sum_{t=T_0+1}^{T_1}2\eta ^2\|x_t\|_{V_{t-1}^{-1}}^2\\
 \leq&\ \sqrt{\sum_{t = T_0 + 1}^{T_1}72\eta^2\|x_t\|_{V_{t-1}^{-1}}^2\Lambda\log\frac{1}{\delta'} } + \sqrt{\sum_{t = T_0 + 1}^{T_1}32\eta^4\|x_t\|_{V_{t-1}^{-1}}^4\log\frac{1}{\delta'}}  + \sum_{t=T_0+1}^{T_1}2\eta^2\|x_t\|_{V_{t-1}^{-1}}^2\\
 \leq&\ \sqrt{\sum_{t = T_0 + 1}^{T_1}72\eta^2\|x_t\|_{V_{t-1}^{-1}}^2\Lambda\log\frac{1}{\delta'} } + \sum_{t = T_0 + 1}^{T_1}6\eta^2\|x_t\|_{V_{t-1}^{-1}}^2\sqrt{\log\frac{1}{\delta'}}  + \sum_{t=T_0+1}^{T_1}2\eta^2\|x_t\|_{V_{t-1}^{-1}}^2 \, .
 \intertext{By~\autoref{lem:sum_x}, we have}
 \leq&\ \sqrt{144\eta d\log\left(1 + \frac{\eta T_1L^2}{d\lambda}\right)\Lambda\log\frac{1}{\delta'}} + 16\eta d\log\left(1 + \frac{\eta T_1L^2}{d\lambda}\right)\sqrt{\log\frac{1}{\delta'}} \, .
\end{align*}
We get what we want from \autoref{prop:concentration}.
\end{proof}

Now we are ready to prove \autoref{thm: bandit main appendix}

\begin{proof}
Let $t_0 = 0, t_1 = T, a_0 = \|\theta_*\|_{V_0}^2$ and
\[
a_1 = \Lambda = 288\max\left\lbrace L_{*}^2\lambda, \frac{d\lambda}{L^2}\log\left(1 + \frac{T}{d}\right)\log\frac{1}{\delta}\right\rbrace \, 
\]
by plugging $\eta = \frac{\lambda}{L^2}$. Now, we invoke \autoref{lem:bandit concentration details} with $A_0=a_0$, $A_1=a_1$, $T_0=0$, $T_1=T$, and $\delta'=\delta$.
Let us verify the two conditions. First, we verify the pull out condition. We have
\begin{align*}
a_{0} + \Delta
&\leq L_{*}^2\lambda + \sqrt{144\frac{d\lambda}{L^2}\log\left(1 + \frac{T}{d}\right)\Lambda\log\frac{1}{\delta}} + 16\frac{d\lambda}{L^2}\log\left(1 + \frac{T}{d}\right)\sqrt{\log\frac{1}{\delta}} \\
&= \frac{\Lambda}{288} + \frac{\Lambda}{\sqrt{2}} + \frac{16\Lambda}{288} < \Lambda \, .
\end{align*}
For the improvement condition, we have $\Lambda_1\leq a_1$ trivially. By \autoref{lem:bandit concentration details} this implies that
\[
\Pr\left[\exists t\in [T],\ X_t > 288\max\left\lbrace L_{*}^2\lambda, \frac{d\lambda}{L^2}\log\left(1 + \frac{T}{d}\right)\log\frac{1}{\delta}\right\rbrace\right] < \delta \, .
\]
\end{proof}
Finally we prove the regret bound~\autoref{thm: bandit main regret appendix}.
\begin{proof}
To prove the regret bound, by setting
$$
\beta_t =  288\max\left\lbrace L_{*}^2\lambda, \frac{d\lambda}{L^2}\log\left(1 + \frac{t}{d}\right)\log\frac{1}{\delta}\right\rbrace \, ,
$$ 
\autoref{thm: bandit main appendix} can guarantee with probability $1-\delta$,  $\theta_* \in B_t$. As a result, defining $\tilde{\theta}_t=\underset{\theta \in B_t}{\text{arg}\max}\underset{x\in D_t}{\max}\langle x,\theta \rangle $, $\left< x_{t}^{*},\theta _* \right> \le \langle x_t,\tilde{\theta}_t \rangle  $.
Therefore,
\[
\langle x_{t}^{*}-x_t,\theta _* \rangle \le \langle x_t,\tilde{\theta}_t-\theta _* \rangle \le \lVert x_t \rVert _{V_{t-1}^{-1}}\lVert \tilde{\theta}_t-\theta _* \rVert _{V_{t-1}}\le 2\lVert x_t \rVert _{V_{t-1}^{-1}}\sqrt{\beta _t} \, .
\]
Taking the sum,
\begin{align*}
R_T &\leq \sum_{i=1}^\top  2
\sqrt{\beta_t}\|x_i\|_{V_{i-1}^{-1}} \, .
\intertext{By Cauchy-Schwarz inequality, we have}
&= 2\sqrt{288T\max\left\lbrace L_{*}^2\lambda, \frac{d\lambda}{L^2}\log\left(1 + \frac{T}{d}\right)\log\frac{1}{\delta}\right\rbrace}\sqrt{\sum_{i=1}^\top \|x_i\|_{V_{i-1}^{-1}}^2}\\
&\overset{\left( i \right)}{\le} 34\sqrt{T\max\left\lbrace L_{*}^2L^2, d\log\left(1 + \frac{T}{d}\right)\log\frac{1}{\delta}\right\rbrace}\sqrt{2d\log\left( 1+\frac{T}{d} \right)}\\
&\leq  34\sqrt{2nT\max\left\lbrace L_{*}^2L^2, d\log\left(1 + \frac{T}{d}\right)\log\frac{1}{\delta}\right\rbrace\log\left( 1+\frac{T}{d} \right)}
\end{align*}
where $(i)$ is by Lemma~\ref{lem:sum_x}.
\end{proof}

\end{document}